\documentclass[12pt,english,pdflatex,twoside]{article}
\usepackage[T1]{fontenc}
\usepackage[latin9]{inputenc}
\usepackage[a4paper,left=27mm,right=27mm,top=34mm,bottom=34mm]{geometry}
\usepackage{textcomp}
\usepackage{amsthm}
\usepackage{amsmath}
\usepackage{amssymb}
\usepackage{esint}
\usepackage{rotating}

\makeatletter

\numberwithin{equation}{section}
\numberwithin{figure}{section}
\theoremstyle{plain}
\newtheorem{thm}{Theorem}[section]
\theoremstyle{remark}
\newtheorem*{rem*}{Remark}
\theoremstyle{plain}
\newtheorem{assumption}[thm]{Assumption}
\theoremstyle{definition}
\newtheorem{defn}[thm]{Definition}
\theoremstyle{plain}
\newtheorem{lem}[thm]{Lemma}

\makeatother

\usepackage{babel}

\begin{document}
\global\long\def\eps{\varepsilon}
\global\long\def\Rr{\mathbb{R}}
\global\long\def\R{\mathbb{R}}
\global\long\def\bbR{\mathbb{R}}
\global\long\def\bbN{\mathbb{N}}
\global\long\def\calT{\mathcal{T}}
\global\long\def\nablaomega{\nabla_{\omega}}
\global\long\def\diveromega{\nabla_{\omega}\cdot}
\global\long\def\symM{\bbR_{s}^{d\times d}}
\global\long\def\diver{\nabla\cdot}
\global\long\def\diverx{\mbox{div}_{x}}
\global\long\def\v{v}
\global\long\def\norm#1{\left\Vert #1\right\Vert }
\global\long\def\Q{Q}
\global\long\def\Y{Y}
\global\long\def\B{B}
\global\long\def\cT{\mathcal{T}}
\global\long\def\cL{\mathcal{L}}
\global\long\def\cH{\mathcal{H}}
\global\long\def\cV{\mathcal{V}}
\global\long\def\Rdd{\bbR^{d\times d}}
\global\long\def\Rnn{\bbR^{n\times n}}
\global\long\def\cF{\mathcal{F}}
\global\long\def\cP{\mathcal{P}}
\global\long\def\cS{\mathcal{S}}
\global\long\def\cB{\mathcal{B}}
\global\long\def\cN{\mathcal{N}}
\global\long\def\cC{\mathcal{C}}
\global\long\def\Nn{\mathbb{N}}
\global\long\def\Zz{\mathbb{Z}}
\global\long\def\n{\boldsymbol{n}}
\global\long\def\dist{\mbox{dist}}
\global\long\def\diam{\mathrm{diam}\,}
\global\long\def\del{\partial}
\global\long\def\dom{\mathrm{dom}\,}
\global\long\def\weakto{\rightharpoonup}
\global\long\def\gkj{\Gamma_{k,j}}
\global\long\def\scp#1#2{\left\langle #1,#2\right\rangle }
\global\long\def\Id{\mathrm{Id}}
\global\long\def\grid{\mathbb{T}}
\global\long\def\Rd{\mathbb{R}^{d}}
\global\long\def\Rn{\mathbb{R}^{n}}
\global\long\def\clos{\mathrm{cl}}
\global\long\def\gridS{\mathbb{S}}
\global\long\def\argmin{\mathrm{argmin}}
\global\long\def\normal{\mathfrak{n}}

\global\long\def\lb{[[}
\global\long\def\rb{]]}
\global\long\def\LOM{L^{2}(\Omega;\Rdd)}
\global\long\def\LOMns{L_{n,s}^{2}(\Omega)}
\global\long\def\LOMn{L_{n}^{2}(\Omega)}
\global\long\def\LOMs{L_{s}^{2}(\Omega)}
\global\long\def\U{U}
\global\long\def\linsymM{\mathcal{L}(\symM,\symM)}
\global\long\def\bbT{\mathbb{T}}
\global\long\def\bbX{\mathbb{X}}
\global\long\def\Ups{\Upsilon}

\title{\bf\Large Stochastic homogenization of plasticity equations}

\pagestyle{myheadings} 
\thispagestyle{plain} 

\markboth{M.\,Heida and B.\,Schweizer}{Stochastic homogenization of
  plasticity equations}

\author{Martin Heida\thanks{Weierstrass Institute, Mohrenstrasse 39,
    D-10117 Berlin, Germany.}\quad and Ben Schweizer\thanks{TU Dortmund,
    Fakult\"at f\"ur Mathematik, Vogelpothsweg 87, D-44227 Dortmund,
    Germany.}}

\date{April 8, 2016}

\maketitle

\vspace*{1mm}
\begin{center}
  \begin{minipage}{14cm}
    {\bf Abstract:} In the context of infinitesimal strain plasticity
    with hardening, we derive a stochastic homogenization result. We
    assume that the coefficients of the equation are random functions:
    elasticity tensor, hardening parameter and flow-rule function are
    given through a dynamical system on a probability space.  A
    parameter $\eps>0$ denotes the typical length scale of
    oscillations. We derive effective equations that describe the
    behavior of solutions in the limit $\eps\to 0$. The homogenization
    limit is based on the needle-problem approach: We verify that the
    stochastic coefficients ``allow averaging'': In average, a strain
    evolution $[0,T]\ni t\mapsto \xi(t) \in \symM$ induces a stress
    evolution $[0,T]\ni t\mapsto \Sigma(\xi)(t) \in \symM$.  With the
    abstract result of \cite{heidaschweizer2014} we obtain the
    stochastic homogenization limit.
  \end{minipage}  
\end{center}
\vspace*{2mm}

\section{Introduction}

In its history, mathematics has often been inspired by questions from
continuum mechanics: Given a body of metal and given a force acting on
it, what is the deformation that the body of metal is experiencing?
Euler has been inspired by this question; much later, the development
of linear and non-linear elasticity theory provided excellent models
(and mathematical theories) for non-permanent deformations. In
contrast, the description of permanent deformations with plasticity
models is much less developed. The only well-established plasticity
models are based on infinitesimal strain theories, ad-hoc
decomposition rules of the strain tensor and flow rules for the
plastic deformation tensor.

Homogenization theory is, in its origins, concerned with the following
question: How does a heterogeneous material (composed of different
materials) behave effectively? Can we characterize an effective
material such that a heterogeneous medium (consisting of a very fine
mixture) behaves like the effective material? This homogenization
question has a positive answer in the context of linear elasticity:
effective coefficients can be computed and bounds for these effective
coefficients are available. The situation is quite different for
plasticity models: Results have been obtained only in the last ten
years.  The effective model cannot be reduced to one macroscopic set
of differential equations. The effective system either remains a
two-scale model or, as we do here, must be formulated with a
hysteretic stress-strain map.
 
With only two exceptions, so far, homogenization results in plasticity
treat essentially the same system: Infinitesimal strains and an
additive decomposition of the strain tensor are used, some hardening
effect is included, and the homogenization is performed in a periodic
setting.  The two exceptions are \cite {FrancfortGiacomini-Homogen}
and \cite {Schweizer09}: In \cite {FrancfortGiacomini-Homogen}, no
hardening effect is used and the limit system is much more
involved. In \cite {Schweizer09}, stochastic coefficients are
permitted, but at the expence of a one-dimensional setting.  The
present article is based on \cite{heidaschweizer2014} and provides the
third exception: We treat a model with stochastic coefficients in
dimensions $2$ and $3$.

We mention at this point the more abstract approach in the framework
of energetic solutions, see \cite {MielkeRS08,MielkeTimofte07}, and
its application in gradient plasticity in \cite {MR2826468}.

\subsubsection*{Plasticity equations}

We study a bounded domain $\Q\subset\Rd$, $d\in \{2,3\}$, occupied by
a heterogeneous material, and its evolution in a time interval $(0,T)
\subset \R$. For a parameter $\eps>0$, we consider on $\Q\times(0,T)$
the plasticity system
\begin{equation}
  \begin{aligned}-\diver\sigma^{\eps} & =f\,, 
    & \sigma^{\eps} & =C_{\eps}^{-1}e^{\eps}\,,\\
    \nabla^{s}u^{\eps} & =e^{\eps}+p^{\eps}\,, & \qquad
    \partial_{t}p^{\eps} & \in\partial\Psi_{\eps}(\sigma^{\eps}-B_{\eps}p^{\eps})\,.
  \end{aligned}
  \label{eq:eps-problem}
\end{equation}
The first relation is the quasi-static balance of forces in the body,
$f$ is a given load, $\sigma$ the stress tensor. The second relation
is Hooke's law which relates linearly the stress $\sigma$ with the
elastic strain $e$. The third relation is the additive decomposition
of the infinitesimal strain $\nabla^{s} u = (\nabla u + (\nabla
u)^T)/2$. The fourth relation is the flow rule for the plastic strain
$p$, it uses the subdifferential $\partial\Psi$ of a convex function
$\Psi$. Kinematic hardening is introduced with the positive tensor
$B_{\eps}$. Hardening is an experimental fact in metals. From the
analytical point of view, hardening simplifies the mathematical
treatment considerably: Standard function spaces can be used, while in
the case without hardening (perfect plasticity) the space $BD(Q)$ of
bounded deformations must be used (measure-valued shear bands can
occur).  We refer to \cite {Alber_book98, HanReddy99} for the
modelling.

Our interest here is to study coefficients $B = B_\eps$ (hardening),
$C = C_\eps$ (elasticity tensor), and $\Psi = \Psi_\eps$ (convex flow
rule function) that depend on the parameter $\eps>0$. We imagine
$\eps$ to be the spatial length scale of the heterogeneities. Since
the coefficients depend on $\eps$, also the solution $(u,\sigma,e,p) =
(u^\eps,\sigma^\eps,e^\eps,p^\eps)$ depends on $\eps$.

We consider only positive and symmetric coefficient tensors, using the
following setting: We denote by $\symM\subset \Rdd$ the space of
symmetric matrices, $\linsymM$ is the space of linear mappings on
$\symM$. For every $\eps>0$ and almost every $x\in \Q$, the tensors
$C_{\eps}(x), B_{\eps}(x) \in \linsymM$ are assumed to be symmetric
with respect to the scalar product on $\symM$. Furthermore, for
constants $\gamma,\beta>0$, we assume the positivity and boundedness
\begin{equation}
  \gamma\left|\xi\right|^{2}\leq\xi:(C_{\eps}(x)\,\xi)
  \leq\frac{1}{\gamma}\left|\xi\right|^{2}\,,
  \qquad\beta\left|\xi\right|^{2}\leq\xi:(B_{\eps}(x)\,\xi)
  \leq\frac{1}{\beta}\left|\xi\right|^{2}\,
  \label{eq:reg-c-b-eps-1}
\end{equation}
for every $\xi\in\symM$, a.e. $x\in\Q$, and every $\eps>0$.

System \eqref {eq:eps-problem} is accompanied by a Dirichlet boundary
condition $u^{\eps}=\U$ on $\partial\Q\times(0,T)$ and an initial
condition for the plastic strain tensor (for simplicity, we assume
here a vanishing initial plastic deformation). Finally, the load $f$
must be imposed. We consider data
\begin{equation} 
  \U\in H^{1}(0,T;H^{1}(\Q;\Rd))\,,\quad 
  f\in H^{1}(0,T;L^{2}(\Q, \Rd))\,,\quad
  p^{\eps}|_{t=0} \equiv 0\,.
  \label{eq:reg-BC-initial}
\end{equation}
The fundamental task of homogenization theory is the following: If
$u^\eps\weakto u$ converges in some topology as $\eps\to 0$, what is
the equation that characterizes $u$?

\subsubsection*{Known homogenization results and the needle-problem
  approach}

The periodic homogenization of system \eqref {eq:eps-problem} was
performed in the last 10 years.  The effective two-scale limit system
was first stated in \cite {Alber_homogenization00}. The rigorous
derivation of the limit system (under different assumptions on the
coefficients) was obtained by Visintin with two-scale convergence
methods \cite{Visintin2005elastplast, Visintin_Kelvin06,
  Visintin_Maxwell09}, by Alber and Nesenenko with phase-shift
convergence \cite{AlberNesenenko09, Nesenenko07}, and by Veneroni
together with the second author with energy methods \cite
{SchweizerVen10}. By the same authors, some progress was achieved
regarding the monotone flow rule and a simplification of proofs in
\cite {SchweizerVeneroni-twoscale}. We refer to these publications
also for a further discussion of the periodic homogenization of system
\eqref {eq:eps-problem}.

The non-periodic homogenization of system \eqref {eq:eps-problem} is
much less treated. In particular, we are not aware of any stochastic
homogenization result (with the exception of \cite {Schweizer09}, but
the analysis of the one-dimensional case is much simpler, since the
stress variable can be obtained by a simple integration from the force
$f$).

For the non-periodic case, a partial homogenization result has been
obtained in \cite{heidaschweizer2014}.  That contribution is based on
the needle-problem approach, which has its origin in \cite
{Schweizer2009needle}. The present article is based on
\cite{heidaschweizer2014} and we therefore describe in the next
paragraph the needle-problem approach in more detail.

In the {\em needle-problem approach,} homogenization is seen as a
two-step procedure. We describe the two steps here with the scalar
model $-\nabla\cdot (a^\eps \nabla u^\eps) = f$ for a deformation
$u^\eps:Q\to \R$.  Step 1 is concerned with cell-problems: One
verifies that, on a representative elementary volume (REV, the unit
square in periodic homogenization) and for a vanishing load, the
material behaves in a well-defined way: An input (here: the averaged
gradient $\xi$ of the solution across the REV) results in a certain
output (here: the averaged stress $\sigma(\xi) = a^* \xi$ for a matrix
$a^*$). Step 2 is concerned with arbitrary domains $Q$ and arbitrary
loads $f$. The conclusion of Step 2 (which can be justified with the
needle-problem approach) is the following: If the REV-analysis
provides the material law $\xi\mapsto \sigma(\xi)$, then the behavior
of the material on the macroscopic scale is characterized by
$-\nabla\cdot (\sigma(\nabla u)) = f$ in $Q$ (in our example by
$-\nabla\cdot (a^* \nabla u) = f$). In \cite {Schweizer2009needle},
these methods are developed and the two-step scheme is illustrated
with the linear model: The assumption of an averaging property on
simplices implies the homogenization on the macroscopic scale with the
corresponding law.

In \cite{heidaschweizer2014}, we performed Step 2 of the
needle-problem approach in the context of plasticity. Our assumption
was that the material parameters allow averaging: solutions on
simplices with affine boundary data $x\mapsto \xi\cdot x$ and
vanishing forces $f\equiv 0$ have convergent stress averages: in the
limit $\eps\to 0$, stress integrals converge to some deterministic
quantity $\Sigma(\xi)$. Due to memory effects in plasticity problems,
one has to find for every evolution of strains $\xi = \xi(t)$ an
evolution of stresses $\Sigma(\xi)(t) = \Sigma(\xi(.))(t)$. In
\cite{heidaschweizer2014}, we derived from this averaging assumption a
homogenization result: For general domains $Q$, general boundary data
$\U$ and general forces $f$, the effective problem for every limit $u
= \lim_{\eps\to 0} u^\eps$ reads
\begin{equation}
  -\diver\Sigma(\nabla^{s}u)=f\quad \text{ in } \Q\times (0,T)\,.
  \label{eq:limit-problem-pre}
\end{equation}

Let us briefly describe the relation between the needle-problem
approach (used here) with classical stochastic homogenization results
(as in \cite{JikovKozlovOleinik94, Kozlov1979,
  papanicolaou1979boundary}: We believe that our result on the
stochastic homogenization of plasticity equations could also be
obtained along the classical route. In such a proof one would first
obtain a two-scale effective problem in the variables
$(x,t,\omega)$. In a second step, one can realize that the dependence
on $x$ can be disintegrated: The two-scale system can be written in
the form \eqref {eq:limit-problem-pre}, if the hysteretic stress
operator $\Sigma$ is defined through a stochastic cell problem in the
variables $(t,\omega)$. In the needle-problem approach, we keep these
two aspects separated: The abstract result ``averaging property for
$\Sigma$ implies homogenization'' of \cite{heidaschweizer2014} is
independent of the stochastic description. The stochastic analysis
concerns only the operator $\Sigma$ and its properties (the work at
hand).

\subsubsection*{The stochastic homogenization result}
\label{sub:The-main-Theorem}

In this contribution, we perform the stochastic homogenization of the
plasticity system. In particular, we demonstrate that the averaging
assumption is satisfied for an evolution operator $\Sigma$ and that
equation \eqref {eq:limit-problem-pre} is the effective plasticity
problem. Comparing with other homogenization results for plasticity
equations, this means that we obtain a {\em disintegrated} effective
system: Equation \eqref {eq:limit-problem-pre} is local in space, it
is not a {\em two-scale} system. The microscopic behavior is
synthesized in the operator $\Sigma$. The only non-local effect occurs
in the time variable, since $\Sigma$ is an evolution operator.

\begin{defn}[The structure of the limit problem]
  \label{def:sto-limit-problem} Let the domain $\Q\subset\Rd$ and the
  time horizon $T>0$ be as above, let $\Omega$ be a probability space
  with ergodic dynamical system as in Section
  \ref{sub:Setting-in-stochastic}, let the stochastic coefficients
  $C$, $B$ and $\Psi$ be as in Assumption \ref{ass:sto-coeffs}. 

  {\bf (i) Definition of the hysteretic strain-to-stress map $\Sigma :
    \xi \mapsto \sigma$.} We consider an input $\xi : [0,T] \to \symM$
  and solve the following stochastic cell problem with a triplet
  $(p,z,\v)$, where $p \in H^{1}(0,T;L^{2}(\Omega;\symM))$, $z\in
  H^{1}(0,T;L^{2}_{sol}(\Omega;\Rdd))$, $\v \in
  H^{1}(0,T;L^{2}_{pot}(\Omega;\Rdd))$, and $z$ is symmetric, $z =
  z^s$:
  \begin{equation}
    \begin{aligned}
      \xi & =Cz-v^s+p & & \mbox{a.e. in }[0,T]\times\Omega\,,\\ \partial_{t}p &
      \in\partial\Psi(z-Bp) & & \mbox{a.e. in }[0,T]\times\Omega\,.
    \end{aligned}
    \label{eq:hom-problem-1}
  \end{equation}
  For the definition of the function spaces $L_{pot}^{2}(\Omega)$ and
  $L_{sol}^{2}(\Omega)$ see \eqref{def:L2pot} and \eqref{def:L2sol}.
  The solution $(p,z,\v)$ defines the operator $\Sigma$,
  \begin{equation}
    \Sigma(\xi)(t) := \int_{\Omega}z(t,\omega)\,   d\cP(\omega)\,.
    \label{eq:hom-problem-2}
  \end{equation}

  {\bf (ii) Definition of the effective equation.}  For boundary data
  $\U$ and loading $f$ as in \eqref {eq:reg-BC-initial}, we search for
  $u\in H^{1}(0,T;H^{1}(\Q))$ such that
  \begin{equation}
    \int_{0}^{T}\int_{\Q} \Sigma(\nabla^{s}u) :\nabla\varphi
    =\int_{0}^{T}\int_{\Q}f\cdot\varphi \quad \forall\varphi\in
    L^{2}(0,T;H_{0}^{1}(\Q))\,.
    \label{eq:hom-problem-3}
  \end{equation}
  Additionally, we demand that the boundary condition $u = \U$ on
  $\partial\Q \times (0,T)$ is satisfied in the sense of traces.
\end{defn}

\begin{rem*}
  The argument of the stress function $\Sigma$ is $\xi = \xi(t)$, in
  the limit problem \eqref {eq:hom-problem-3} the stress function is
  evaluated, for every $x\in Q$, with the argument $\xi(.) =
  \nabla_{x}^{s} u(.,x)$.  For a more detailed description of the
  limit problem \eqref {eq:hom-problem-3} see Definition \ref
  {def:limit-problem}.  The precise statement of the stochastic
  cell-problem \eqref {eq:hom-problem-1} and the corresponding
  definition of the operator $\Sigma$ in \eqref {eq:hom-problem-2} is
  given in Definition \ref {def:Sigma}.
\end{rem*}

Our stochastic homogenization result follows by applying the main
theorem of \cite{heidaschweizer2014}. Essentially, we only have to
verify that, if the coefficient functions of system
\eqref{eq:eps-problem} are given by an ergodic stochastic process,
then the coefficients ``allow averaging'': In the limit $\eps\to 0$,
averages of the stress (for a homogeneous plasticity system on a
simplex with affine boundary data $\xi$) are given by the operator
$\Sigma$.

We verify this statement in Sections \ref {sec.stoch-cell} and \ref
{sec.proof-main}. The consequence is the following homogenization
theorem, which is our main result.

\begin{thm}[Stochastic homogenization in plasticity]
  \label{thm:Main-Theorem} Let $\Q\subset\Rd$ be a bounded domain,
  $d\in \{2,3\}$, $T>0$. Let $\tau$ be an ergodic dynamical system on
  the probability space $(\Omega,\Sigma_\Omega,\cP)$ as in Section
  \ref{sub:Setting-in-stochastic}, let the stochastic coefficients
  $B$, $C$, $\Psi$ and the data $\U$ and $f$ be as in Assumption
  \ref{ass:sto-coeffs}.  Then, there exists a unique solution $u$ to
  the limit problem \eqref {eq:hom-problem-1}--\eqref
  {eq:hom-problem-3} of Definition \ref {def:sto-limit-problem}.  For
  $\omega\in\Omega$, let $(u^\eps,\sigma^\eps,e^\eps,p^\eps)$ be weak
  solutions to \eqref{eq:eps-problem}.  Then, for
  a.e. $\omega\in\Omega$, as $\eps\rightarrow0$,
  \begin{align*}
    &u^{\eps}\weakto u \quad\mbox{ weakly in }
    H^{1}(0,T;H^{1}(\Q))\mbox{ and }\\
    &\sigma^{\eps}\weakto \Sigma(\nabla^{s}u) \quad\mbox{ weakly in
    }H^{1}(0,T;L^{2}(\Q))\,.\end{align*}
\end{thm}

\begin{rem*}
  The weak solution concept for the $\eps$-problem
  \eqref{eq:eps-problem} is made precise in Definition
  \ref{def:Sol-eps-problem}. The unique existence of a solution
  $u^\eps$ for a.e.\,$\omega\in \Omega$ is guaranteed by Theorem
  \ref{thm:uniqueness-eps-problem}.
\end{rem*}

The proof of Theorem \ref {thm:Main-Theorem} is concluded in Section
\ref{ssec:Proof-of-Main-Theorem}.  A sketch of the proof is presented
at the end of Section \ref{ssec:The-Needle-Problem}.

\subsection{Setting in stochastic homogenization}
\label{sub:Setting-in-stochastic}

We follow the traditional setting in stochastic homogenization, first
outlined by Papanicolaou and Varadhan in
\cite{papanicolaou1979boundary} and by Kozlov in \cite{Kozlov1979},
later used by Jikov, Kozlov and Oleinik \cite{JikovKozlovOleinik94}.
Let $(\Omega,\Sigma_{\Omega},\cP)$ be a probability space where we
assume that the $\sigma$-algebra $\Sigma_{\Omega}$ is countably
generated. This implies that $L^2(\Omega)$ is separable.  Let
$(\tau_{x})_{x\in\Rd}$ be an ergodic dynamical system on
$(\Omega,\Sigma_{\Omega},\cP)$. We rely on the following definitions:
A family $(\tau_{x})_{x\in\Rd}$ of measurable bijective mappings
$\tau_{x}:\Omega\mapsto\Omega$ is called a dynamical system on
$(\Omega,\Sigma_{\Omega},\cP)$ if it satisfies
\begin{enumerate}
\item [(i)]$\tau_{x}\circ\tau_{y}=\tau_{x+y}$ , $\tau_{0}=id$ $\quad$ (group
  property)
\item [(ii)]$\cP(\tau_{-x}B)=\cP(B)\quad\forall x\in\Rd,\,\,
  B\in\Sigma_{\Omega}$ $\quad$(measure preservation)
\item [(iii)]$A:\,\,\Rd\times\Omega\rightarrow\Omega\qquad(x,\omega)
  \mapsto\tau_{x}\omega$ is measurable $\quad$(measurability property)
\end{enumerate}
We say that the system $(\tau_{x})_{x\in\Rd}$ is ergodic, if for every
measurable function $f:\Omega\rightarrow\Rr$ holds
\begin{equation}
  \begin{split}
    \left[f(\omega)\stackrel{}{=}f(\tau_{x}\omega)\,\,\forall
      x\in\Rd\,,\, a.e.\,\,\omega\in\Omega\right] &
    \Rightarrow
    \left[\exists c_0\in \R:\ f(\omega)=c_0\,\,\textnormal{for a.e.}\,\,
      \omega\in\Omega\right]\,.
  \end{split}
  \label{eq:def_ergodicity}
\end{equation}
Given $f\in L^2(\Omega)$ and $\omega\in\Omega$, we call $f_\omega:\,\Rn\to\R$,
$x\mapsto f(\tau_x \omega)$ the $\omega$-realization of $f$.
An important property of ergodic dynamical systems is the fact that
spatial averages can be related to expectations. For a quite general
version of the ergodic theorem, we refer to \cite{Zhikov2006}. 
The following simple version is sufficient for our purposes.
\begin{thm}[Ergodic theorem]
  \label{thm:ergodic-thm} Let $(\Omega,\Sigma_{\Omega},\cP)$ be a
  probability space with an ergodic dynamical system
  $(\tau_{x})_{x\in\Rd}$ on $\Omega$. Let $f\in L^{1}(\Omega)$ be a
  function and $\Q\subset\Rd$ be a bounded open set. Then, for
  $\cP$-almost every $\omega\in\Omega$,
  \begin{equation}
    \lim_{\eps\rightarrow0}\int_{\Q}f(\tau_{x/\eps}\omega)\,
    dx=\lim_{\eps\rightarrow0}\int_{\Q}f_\omega\left(\frac{x}{\eps}\right)\,
    dx=\left|\Q\right|\int_{\Omega}f(\omega)\, d\cP(\omega)\,.
    \label{eq:ergodic-thm}
  \end{equation} 
  Furthermore, for every $f\in L^{p}(\Omega)$, $1\leq p\leq\infty$,
  and a.e. $\omega\in\Omega$, the function $f_\omega(x) =
  f(\tau_{x}\omega)$ satisfies $f_\omega\in L_{loc}^{p}(\Rd)$.  For
  $p<\infty$ holds $f_\omega(\cdot / \eps) =
  f(\tau_{\cdot/\eps}\omega) \weakto\int_{\Omega}f\, d\cP$ weakly in
  $L_{loc}^{p}(\Rd)$ as $\eps\rightarrow0$.
\end{thm}

For brevity of notation in calculations and proofs, we will often omit
the symbol $d\cP$ in $\Omega$-integrals. We assume that the
coefficients in \eqref{eq:eps-problem} have the form
\begin{equation}
  \label{eq:C-eps-from-C}
  C_{\eps}(x) = C(\tau_{\frac{x}{\eps}}\omega)\,,\qquad B_{\eps}(x) =
  B(\tau_{\frac{x}{\eps}}\omega)\,,\qquad \Psi_{\eps}(\sigma) =
  \Psi(\sigma;\tau_{\frac{x}{\eps}}\omega) 
\end{equation}
for some functions $B$, $C$, and $\Psi$, see Assumption
\ref{ass:sto-coeffs}.

Using the function spaces
\begin{align*} 
  L_{pot,loc}^{2}(\Rd) & :=  \left\{ u\in
    L_{loc}^{2}(\Rd;\Rdd)\,\,|\,\,\forall U\,\,\mbox{bounded domain,
    }\exists\varphi\in H^{1}(U;\Rd)\,:\, u=\nabla\varphi\right\},\\
  L_{sol,loc}^{2}(\Rd) & :=  \left\{ u\in
    L_{loc}^{2}(\Rd;\Rdd)\,\,|\,\,\int_{\Rd} u\cdot\nabla\varphi=0\,\,
    \forall\varphi\in C_{c}^{1}(\Rd)\right\},
\end{align*}
we follow Chapter 7 in \cite{JikovKozlovOleinik94} and define
\begin{align}
  L_{pot}^{2}(\Omega)&:=\left\{v\in L^2(\Omega;\Rdd)\,|\, x\mapsto v(\tau_x\omega) 
    \text{ in } L_{pot,loc}^{2}(\Rd)\text{ for a.e. }\omega\in\Omega\right\}\label{def:L2pot},\\
  \cV^2_{pot}(\Omega)&:=\left\{f\in L_{pot}^{2}(\Omega)\,|\,\int_\Omega f\,d\cP=0\right\},\\
  L_{sol}^{2}(\Omega)&:=\left\{v\in L^2(\Omega;\Rdd)\,|\, x\mapsto
    v(\tau_x\omega) \text{ in } L_{sol,loc}^{2}(\Rd)\text{ for
      a.e. }\omega\in\Omega\right\}\label{def:L2sol}\,.
\end{align}
The three spaces \eqref{def:L2pot}--\eqref{def:L2sol} are closed
subspaces of $L^2(\Omega;\Rdd)$. The latter spaces can be decomposed
in an orthogonal sum as $L^2(\Omega;\Rdd)=\cV^2_{pot}(\Omega)\oplus
L_{sol}^{2}(\Omega)$, see \cite{JikovKozlovOleinik94}.

\begin{rem*}
  The periodic homogenization setting is a special case of the
  stochastic setting, and we recover known results in the periodic
  case. The cell problem on the periodicity cell is encoded in
  \eqref{eq:hom-problem-1} with the help of the spaces
  $L_{pot}^{2}(\Omega)$ and $L_{sol}^{2}(\Omega)$ ($\v^s$ is a
  symmetrized gradient and $z$ has a vanishing divergence).
\end{rem*}

\subsection{Solution concepts and existence results}

To formulate a stochastic setting, we consider $C,B\in
L^{\infty}(\Omega; \linsymM)$, pointwise symmetric, such that for
$\gamma, \beta >0$ holds
\begin{equation}
  \gamma\left|\xi\right|^{2}\leq\xi:C(\omega)\xi\leq\frac{1}{\gamma}\left|\xi\right|^{2}\,,
  \qquad\beta\left|\xi\right|^{2}\leq\xi:B(\omega)\xi\leq\frac{1}{\beta}\left|\xi\right|^{2}\,,
  \label{eq:ass-data-1}
\end{equation} 
for every $\xi\in\Rd$ and a.e. $\omega\in\Omega$.  Let $\Psi\,:\,\symM
\times \Omega\to (-\infty,+\infty]$, $(\xi,\omega)\mapsto\Psi(\xi,\omega)$ 
be measurable in $\symM\times
\Omega$, lower semicontinuous and convex in $\symM$ for
a.e. $\omega\in\Omega$, and with $\Psi(0, \omega) = 0$ for
a.e.\,$\omega\in \Omega$. We furthermore assume that for
a.e.\,$\omega\in\Omega$ there is $c(\omega)>0$ such that the convex
dual (in the first variable) satisfies
\begin{equation}
  \left|\Psi^{\ast}(\sigma;\tau_{x}\omega)-\Psi^{\ast}(\sigma;\tau_{y}\omega)\right|\leq
  c(\omega)\,\left|x-y\right|\left|\sigma\right|\qquad\forall\sigma\in\symM\,,\ 
  x,y\in\Rd\,.\label{eq:psi-ast-sto}
\end{equation}

We note that the above assumption on $\Psi$ implies that no
discontinuities are allowed in the flow rule.

\begin{assumption}[Data]
  \label{ass:sto-coeffs} Let $C,B\in L^{\infty}(\Omega;\linsymM)$ and
  $\Psi\,:\,\symM\times\Omega\rightarrow(-\infty,+\infty]$ satisfy 
	\eqref{eq:ass-data-1}--\eqref{eq:psi-ast-sto}.
  We consider only parameters $\omega\in\Omega$ such that the
  $\omega$-realizations $C_{\omega}(x):=C(\tau_{x}\omega)$,
  $B_{\omega}(x):=B(\tau_{x}\omega)$ are measurable and such that
  \eqref {eq:reg-c-b-eps-1} and \eqref{eq:psi-ast-sto} hold. We
  furthermore assume that $\U$ and $f$ satisfy the regularity \eqref
  {eq:reg-BC-initial} and the compatibility conditions $\U|_{t=0} =
  0$, $f|_{t=0} = 0$.
\end{assumption}

Our aim is to study \eqref {eq:eps-problem} with the coefficients
defined in \eqref {eq:C-eps-from-C}.  By slight abuse of notation and
omitting the index $\omega$ whenever possible, we also write
$C_{\eps}(x) := C_{\eps,\omega}(x) := C(\tau_{\frac{x}{\eps}}\omega)$
and $B_{\eps}(x) := B_{\eps,\omega}(x) :=
B(\tau_{\frac{x}{\eps}}\omega)$ as in \eqref {eq:C-eps-from-C}. We
assume that they satisfy \eqref{eq:reg-c-b-eps-1} and that
$\Psi_{\eps}$ satisfies
\begin{equation}
  \left|\Psi_{\eps,\omega}^{\ast}(\sigma;x_{1})-\Psi_{\eps,\omega}^{\ast}(\sigma;x_{2})\right|\leq
  c(\eps,\omega)\left|x_{1}-x_{2}\right|\left|\sigma\right|\,.
  \label{eq:lipschitz-psi-1-eps}
\end{equation}
This condition is of a technical nature. It is used only in the proof
of the existence result of Theorem \ref{thm:uniqueness-eps-problem}.
We remark that the existence result remains valid also without
assumption \eqref{eq:lipschitz-psi-1-eps}, as can be shown with the
methods of Section \ref {sec.stoch-cell}. Since we do not want to
repeat the proof of Theorem \ref {thm:uniqueness-eps-problem} here, we
assume the above Lipschitz condition.

\begin{defn}[Weak formulation of the $\eps$-problem]
  \label{def:Sol-eps-problem} We say that $(u^{\eps}, \sigma^\eps,
  e^{\eps}, p^{\eps})$ is a weak solution to the $\eps$-problem
  \eqref{eq:eps-problem} on $\Q$ with boundary condition $\U$ if the
  following is satisfied: There holds $u^{\eps}=v^{\eps}+\U$ with \[
  v^{\eps}\in H^{1}(0,T;H_{0}^{1}(\Q))\,,\quad
  e^{\eps},p^{\eps},\sigma^{\eps}\in H^{1}(0,T;L^{2}(\Q;\symM))\,,\]
  equation $-\diver\sigma^{\eps} =f$ of \eqref{eq:eps-problem} holds
  in the distributional sense and the other relations of
  \eqref{eq:eps-problem} hold pointwise almost everywhere in $\Q\times
  (0,T)$.
\end{defn}

We note that, due to the regularity of $\sigma^{\eps}$, every weak
solution to \eqref{eq:eps-problem} satisfies
\begin{align}
  \int_{0}^{T}\int_{\Q}\sigma^{\eps}:\nabla^{s}\varphi &
  =\int_{0}^{T}\int_{\Q}f\cdot\varphi\qquad\forall\varphi\in
  L^{2}(0,T;H_{0}^{1}(\Q))\,.\label{eq:eps-weak-1}
\end{align} 
Theorem 1.2 of \cite{heidaschweizer2014} provides the following
existence result.

\begin{thm}[Existence of solutions to the $\eps$-problem]
  \label{thm:uniqueness-eps-problem} Let the coefficient functions
  $C$, $B$, $\Psi$, the parameter $\omega\in\Omega$, and the data $\U$
  and $f$ be as in Assumption \ref{ass:sto-coeffs}. Then, for every
  $\eps>0$, there exists a unique weak solution $(u^{\eps},
  \sigma^\eps, e^{\eps}, p^{\eps})$ to the $\eps$-problem
  \eqref{eq:eps-problem} in the sense of Definition
  \ref{def:Sol-eps-problem}. The solutions satisfy the a priori
  estimate
  \begin{equation}
   \norm{u^{\eps}}_{\cV^{1}_{1}}
    + \norm{e^{\eps}}_{\cV^{1}_{0}} 
    + \norm{p^{\eps}}_{\cV^{1}_{0}}
    + \norm{\sigma^{\eps}}_{\cV^{1}_{0}}\leq C\,,
    \label{eq:apriori-eps-prob}
  \end{equation} 
  in the spaces $\cV^{1}_{0} := H^1(0,T; L^{2}(\Q;\symM))$ and
  $\cV^{1}_{1} := H^1(0,T;H_{0}^{1}(\Q))$, the constant $C =
  C(\U,f,\beta,\gamma)$ depends on $\beta$ and $\gamma$ from
  \eqref{eq:reg-c-b-eps-1}, but it does not depend on $\eps>0$ or
  $\omega\in \Omega$.
\end{thm}

\subsection{The needle problem approach to plasticity}
\label{ssec:The-Needle-Problem}

The main result of \cite{heidaschweizer2014} is a homogenization
theorem. Under the assumption that causal operators $\Sigma$ and $\Pi$
satisfy certain admissibility and averaging properties, we obtain the
convergence of the $\eps$-solutions $u^\eps$ to the solution $u$ of
the effective problem \eqref{eq:limit-problem-pre}. We next recall the
required properties. In the following, we use the space
$H^{1}_*(0,T;\symM) := H^{1}(0,T;\symM) \cap \{ \xi\ |\ \xi|_{t=0} =
0\}$ of evolutions with vanishing initial values.

\begin{defn}[Averaging]
  \label{def:Averaging} We say that a map
  $F:H^{1}_*(0,T;\symM)\rightarrow H^{1}(0,T;\symM)$ defines a {\em
    causal operator,} if, for almost every $t\in [0,T]$, the value
  $F(\xi,t):=F(\xi)(t)$ is independent of $\xi|_{(t,T]}$.  We say that
  the coefficients $C_{\eps}$, $B_{\eps}$ and $\Psi_{\eps}$ {\em allow
    averaging,} if there exist causal operators $\Sigma$ and $\Pi$
  such that the following property holds: For every simplex
  $\calT\subset Q$, every boundary condition $\xi\in
  H^{1}_*(0,T;\symM)$ and every additive constant $a\in
  H^{1}(0,T;\Rd)$, the corresponding solution $(u^{\eps}, \sigma^\eps,
  e^{\eps}, p^{\eps})$ of the $\eps$-problem \eqref{eq:eps-problem} on
  $\cT$ with $f=0$ and $\U(x,t)=\xi(t)x+a(t)$ satisfies the following:
  As $\eps\to 0$, for a.e. $t\in (0,T)$, the averages of $p^{\eps}$
  and $\sigma^{\eps}$ converge:
  \begin{equation}
    \fint_{\calT}p^{\eps}(t)\rightarrow\Pi(\xi)(t)\,,\qquad
    \fint_{\calT}\sigma^{\eps}(t)\rightarrow\Sigma(\xi)(t)\,.
    \label{eq:averaging-conv-p-sig}
  \end{equation}
  Here, $\fint_{\calT} = |\calT|^{-1} \int_{\calT}$ denotes averages.
  In particular, we demand that limits of (averages of) stress and
  plastic strain depend only on the (time-dependent) boundary
  condition $\xi$, not on $a$ and not on the simplex $\cT$.
\end{defn}

\begin{defn}[Effective equation in the needle problem approach]
  \label{def:limit-problem}The effective plasticity problem in the
  needle problem approach is given by
  \begin{equation}
    -\diver\Sigma(\nabla^{s}u)=f\quad 
    \text{ in } Q\times (0,T)\,,\label{eq:limit-problem}
  \end{equation}
  with boundary condition $u = \U$ on $\partial Q\times (0,T)$.  A
  function $u$ is a solution to this limit problem if $u=\U+v$ holds
  with $\v\in H^{1}(0,T;H_{0}^{1}(\Q;\Rd))$ and
  \eqref{eq:limit-problem} is satisfied in the distributional
  sense. Regarding the expression $\Sigma(\nabla^{s}u)$ we note that,
  for a.e.\,$x\in Q$, the map $t\mapsto \nabla^{s}u(x,t)$ is in the
  space $H^1_*(0,T;\symM)$, hence $\Sigma(\nabla^{s}u)$ is
  well-defined for almost every point in $Q\times (0,T)$.
\end{defn}

\paragraph{Result of the needle problem approach.}

In Theorem 1.6 of \cite{heidaschweizer2014}, the abstract operator
$\Sigma$ is assumed to satisfy two conditions: (i) Averaging
property. This assumption is recalled in Definition
\ref{def:Averaging}. (ii) Admissibility. Admissibility is defined in
Definition 1.5 of \cite{heidaschweizer2014} as: The effective problem
has a solution.

The existence property of the admissibility condition (ii) can be
shown by proving that Galerkin approximations converge to
solutions. We formulate a sufficient condition in this spirit in
Definition \ref{def:Averaging} below. We therefore obtain from Theorem
1.6 of \cite{heidaschweizer2014}:

\begin{thm}[Needle-approach homogenization theorem in plasticity]
  \label{thm:needle-thm} Let $\Q\subset\Rd$ be open and bounded, let
  the data $f$ and $\U$ be as in Assumption \ref {ass:sto-coeffs}, let
  the coefficients $C_{\eps}$, $B_{\eps}$ and $\Psi_{\eps}$ be as
  above, satisfying \eqref {eq:reg-c-b-eps-1}. Let the data allow
  averaging in the sense of Definition \ref{def:Averaging} with causal
  operators $\Sigma$ and $\Pi$, and let $\Sigma$ satisfy the
  admissibility condition of Definition \ref
  {def:admissible-operator}.  Let $(u^\eps,\sigma^\eps,e^\eps,p^\eps)$
  be the weak solutions to the $\eps$-problems
  \eqref{eq:eps-problem}. Then, as $\eps\rightarrow0$, there holds
  \begin{align*}
    &u^{\eps}\weakto u\quad\mbox{weakly in }H^{1}(0,T;H_{0}^{1}(\Q;\Rd))\,,\\
    &p^{\eps}\weakto\Pi(\nabla^{s}u),\quad\sigma^{\eps}\weakto\Sigma(\nabla^{s}u)
    \quad\mbox{weakly in }H^{1}(0,T;L^{2}(\Q;\Rdd))\,,
  \end{align*} 
  where $u$ is the unique weak solution to the homogenized problem
  \[ -\diver\Sigma(\nabla^{s}u)=f\qquad\text{ on } \Q\times (0,T)\] 
  with boundary condition $\U$ in the sense of Definition
  \ref{def:limit-problem}.
\end{thm}

\paragraph{An assumption that implies admissibility.} For arbitrary
$h>0$, we use a polygonal domain $\Q_{h}\subset\Q$ and a triangulation
$\grid_{h}$ with the properties
\begin{equation}
  \begin{split} & \grid_{h}:=\{\cT_{k}\}_{k\in\Lambda_{h}}
    \quad\mbox{is a triangulation of }\Q_{h},\quad\diam(\cT_{k})<h
    \quad\forall\,\cT_{k}\in\grid_{h},\\
    & \Q_{h}\text{ has the property that }
    x\in\Q,\dist(x,\partial\Q)\ge h\text{ implies }x\in\Q_{h}\,,\end{split}
  \label{eq:discretproperty}
\end{equation}
where $\cT_{k}$ are disjoint open simplices and
$\Lambda_{h}\subset\Nn$ is a finite set of indices. We always assume
that the sequence of meshes is regular in the sense of \cite{Ciarlet},
Section 3.1. As in \cite{Schweizer2009needle}, we consider the finite
element space of continuous and piecewise linear functions with
vanishing boundary values, 
\begin{equation} 
  Y_{h}:=\left\{ \phi\in
    H_{0}^{1}(\Q)\,|\ \phi|_{\cT_{k}}\ \mbox{is affine  }
    \forall\,\cT_{k}\in\grid_{h},\ \phi\equiv0\text{ on
    }\Q\setminus\Q_{h}\right\} \,.\label{eq:def-Y-h}
\end{equation}

\emph{Discretization of boundary conditions:} We may extend the
triangulation of $\Q_{h}$ by a finite amount of simplices with
diameter not greater than $h$ to obtain a grid $\tilde{\grid}_{h}$
that covers $\Q$ in the sense
$\Q\subset\bigcup_{\cT_{k}\in\tilde{\grid}_{h}} \bar\cT_{k}$ and
introduce the finite element space $\tilde{Y}_{h}:=\left\{ \phi\in
  H^{1}(\Q)\,|\ \phi|_{\cT_{k}\cap\Q}\ \mbox{is affine
  }\forall\,\cT_{k}\in\tilde{\grid}_{h}\right\} \,.$ Denoting by
$R_{\Q,h}$ the $H^{1}$-orthogonal Riesz-projection
$H^{1}(\Q)\rightarrow\tilde{Y}_{h}$, we set $\U_{h}:=R_{\Q,h}(\U)$ and
observe that $\U_{h}\rightarrow\U$ converges strongly in
$H^{1}(0,T;H^{1}(\Q))$ as $h\rightarrow0$.

\begin{defn}[Sufficient condition for admissibility of $\Sigma$]
  \label{def:admissible-operator} We consider a causal operator
  $\Sigma: H^{1}_*(0,T;\symM)\to H^{1}(0,T;\symM)$. We say that
  $\Sigma$ {\em satisfies the sufficient condition for admissibility}
  if the following property holds: Let $h\rightarrow0$ be a sequence
  of positive numbers, let $\grid_{h}$ be a sequence of regular grids
  satisfying \eqref{eq:discretproperty}, and let $v_{h}\in
  L^{2}(0,T;Y_{h})$ be a corresponding sequence of solutions to the
  discretized problems (the existence is guaranteed in
  \cite{heidaschweizer2014})
  \[
  \int_{\Q}\Sigma\left(\nabla^{s}\left(v_{h}+\U_{h}\right)\right):
  \nabla\varphi_h = \int_{\Q}f\varphi_h \qquad\forall\varphi_h\in
  L^{2}(0,T;Y_{h})\,.\] Assume furthermore that the solutions
  converge, $v_{h}\weakto v$ weakly in $H^{1}(0,T;H_{0}^{1}(\Q))$ as
  $h\to 0$.  Then $v$ is a solution to
  \[
  \int_{\Q}\Sigma\left(\nabla^{s}\left(v+\U\right)\right):\nabla\varphi=\int_{\Q}f\varphi
  \qquad\forall\varphi\in L^{2}(0,T;H_{0}^{1}(\Q))\,.\]
\end{defn}

\paragraph{Remaining program.}
Using Theorem \ref{thm:needle-thm}, our stochastic homogenization
result of Theorem \ref{thm:Main-Theorem} can be shown as follows: For
stochastic parameters $C_{\eps}$, $B_{\eps}$ and $\Psi_{\eps}$ we
define causal operators $\Sigma$ and $\Pi$ with cell-problems on
$\Omega$. For these operators, we only have to check the averaging
property of Definition \ref{def:Averaging} and the admissibility
condition of Definition \ref{def:admissible-operator}.

\section{Stochastic cell problem and definition of $\Sigma$}
\label{sec.stoch-cell}

Given a strain evolution $\xi$, we want to define the corresponding
evolution $\Sigma(\xi)$ of plastic stresses. For the strain $\xi$, we
use the function space 
\begin{equation}
  H^{1}_*(0,T;\symM) := \left\{ \xi \in H^{1}(0,T;\symM)\ \left|\ \xi|_{t=0} = 0\right. \right\}
  \label{eq:strain-space}
\end{equation}
of evolutions with vanishing initial values. For any function $\xi\in
H^{1}_*(0,T;\symM)$ we consider the ordinary differential equation
(inclusion) for $p(t,\,.\,)\in L^{2}(\Omega;\symM)$,
\begin{equation}
  \partial_{t}p(t,\omega)
  \in\partial\Psi\left(z(t,\omega)-B(\omega)\, p(t,\omega)\,;\,\omega\right)
  \label{eq:EPISIG-1}
\end{equation}
(equality pointwise a.e.), with the initial condition
$p(0,\omega)=0$. In order to close the system, the function $z(t)$
must be determined through $\xi(t)$ and $p(t)$. We search for a map
$z(t) \in L_{sol}^{2}(\Omega)$, symmetric in every point $\omega$,
i.e. $z(t,\omega) = z^T(t,\omega)$, such that the equality
\begin{equation}
  C z(t) = \xi(t) + \v^s(t) - p(t)\,\label{eq:EPISIG-2}
\end{equation}
holds in $L^{2}(\Omega)$ for a function $\v\in
L^2(0,T;\cV_{pot}^{2}(\Omega))$.  Throughout this text we use $z^s = (z
+ z^T)/2$ for the symmetric part of a matrix $z$; for the symmetric
matrix $z$ there holds $z = z^s$.  Note that $\v\in
\cV_{pot}^{2}(\Omega)$ does not imply $\v^s\in \cV_{pot}^{2}(\Omega)$.  Up
to the matrix factor $C$ and the symmetrization, equation
\eqref{eq:EPISIG-2} is a Helmholz decomposition of the field $\xi(t) -
p(t)$: Essentially, the given field is decomposed into a gradient
field and a solennoidal field. It is therefore plausible that, given
$\xi(t)$ and $p(t)$, \eqref{eq:EPISIG-2} yields $z(t)$ and thus closes
the evolution equation \eqref{eq:EPISIG-1}.  The rigorous existence
result is provided in the following theorem.

\begin{thm}
  \label{thm:E-Pi-Sig-well-def} Let $C$, $B$ and $\Psi$ be as in
  Assumption \ref{ass:sto-coeffs}. Then, for $\xi\in
  H^{1}_*(0,T;\symM)$, there exists a unique solution $(p,z,\v) \in
  H^{1}(0,T;L^{2}(\Omega;\symM))\times
  H^{1}(0,T;L^{2}_{sol}(\Omega;\Rdd)) \times
  H^{1}(0,T;\cV^{2}_{pot}(\Omega;\Rdd))$ with $z = z^s$ to
  \eqref{eq:EPISIG-1}--\eqref{eq:EPISIG-2} satisfying the a priori
  estimate
  \begin{equation} 
    \norm p_{\cV^{1}_{0}} + \norm  z_{\cV^{1}_{0}} + \norm{\v}_{\cV^{1}_{0}}\leq
    C \norm{\xi}_{H^{1}(0,T)}\,,\label{eq:EPISIG-apriori-global}
  \end{equation}
  where $\cV^{1}_{0}:=H^{1}(0,T;L^{2}(\Omega;\symM))$. The solution
  $(p,z,\v)\in(\cV^{1}_{0})^{3}$ depends continuously on $\xi\in
  H^{1}_*(0,T;\symM)$ with respect to the weak topologies in both
  spaces.
\end{thm}

Theorem \ref {thm:E-Pi-Sig-well-def} permits us to define the
operators $\Sigma$ and $\Pi$.
\begin{defn}[The effective plasticity operators]\label{def:Sigma}
  For arbitrary $\xi\in H^{1}_*(0,T;\symM)$, let $(p,z,\v)$ be the
  solution of \eqref{eq:EPISIG-1}--\eqref{eq:EPISIG-2} with $z =
  z^s$. We set
  \begin{equation}
    \Sigma(\xi)(t):=\int_{\Omega}z(t,\omega)\, d\cP(\omega)\,,\qquad
    \Pi(\xi)(t):=\int_{\Omega}p(t,\omega)\, d\cP(\omega)\,.\label{eq:def-sig-stoch}
  \end{equation}
  We note that the operators $\Sigma,\,\Pi: H^{1}_*(0,T;\symM) \to
  H^{1}(0,T;\symM)$ are well defined and continuous by Theorem
  \ref{thm:E-Pi-Sig-well-def}.
\end{defn}

The rest of this section is devoted to the proof of Theorem
\ref{thm:E-Pi-Sig-well-def}.  We proceed as follows: In Section
\ref{sub:Galerkin}, we introduce a Galerkin approximation scheme for
\eqref{eq:EPISIG-1}--\eqref{eq:EPISIG-2}, using additionally a
regularization of $\Psi$. In \ref{sub:Convex-functionals}, we recall
some results from the theory of convex functions, in
\ref{sub:Korn-inequality} we provide a Korn's inequality in the
probability space $\Omega$. In Section
\ref{sub:Approximation-procedure} we prove existence and uniqueness of
solutions to the approximate problems and show that these solutions
satisfy uniform bounds. Finally, in Section
\ref{sub:Proof-of-Thm-Main2}, we show that the solutions of the
approximate problems converge to the unique solution of the original
system \eqref{eq:EPISIG-1}--\eqref{eq:EPISIG-2}.

\subsection{\label{sub:Galerkin}Galerkin method and regularization}

\paragraph*{Finite dimensional approximation.}

In what follows, let $\left\langle \varphi,\psi\right\rangle
_{\Omega}:=\int_{\Omega}\varphi:\psi\,d\cP$ denote the scalar product
in $L^{2}(\Omega):=L^{2}(\Omega;\Rdd)$. We choose complete orthonormal
systems $\left\{ e_{k}\right\} _{k\in\Nn}$ of $\cV_{pot}^{2}(\Omega)$
and $\left\{ \tilde{e}_{k}\right\} _{k\in\bbN}$ of
$L_{sol}^{2}(\Omega)$ and consider the finite dimensional spaces
\begin{align*}
  &\tilde{L}_{n}^{2}(\Omega) :=\mbox{span}\left\{ e_{k}\right\}
  _{k=1,\dots,n}\oplus\mbox{span}\left\{ \tilde{e}_{k}\right\}
  _{k=1,\dots,n}\,, & \LOMn & :=\tilde{L}_{n}(\Omega)\oplus\left\{
    v^s\,|\,
    v\in\tilde{L}_{n}(\Omega)\right\} \,,\\
  &\cV_{pot,n}^{2}(\Omega)  :=\cV_{pot}^{2}(\Omega)\cap\LOMn\,, &
  L_{sol,n}^{2}(\Omega) & :=L_{sol}^{2}(\Omega)\cap\LOMn\,.
\end{align*} 
We furthermore set $\LOMs:=L^{2}(\Omega;\symM)$ and $\LOMns:=\left\{
  v^{s}\,|\, v\in\LOMn\right\} $. Since constants are in
$L_{sol}^{2}(\Omega)$, we can assume that they are in $\LOMn$ and thus
in $L_{sol,n}^{2}(\Omega)$ for every $n\geq d^2$. We finally introduce
the orthogonal projection $P_{n}:\,L^2(\Omega)\rightarrow\LOMn$ and
note that $P_{n}\varphi\rightarrow\varphi$ strongly in $\LOM$ as
$n\rightarrow\infty$ for every $\varphi\in\LOM$.

\paragraph*{Definition of regularized convex functionals.}

In order to prove Theorem \ref{thm:E-Pi-Sig-well-def}, we consider the
family of Moreau-Yosida approximations\begin{equation}
\Psi^{\delta}(\sigma,\omega):=\inf_{\xi\in\symM}\left\{
\Psi(\xi,\omega)+\frac{\left|\xi-\sigma\right|^{2}}{2\delta}\right\}
\,,\label{eq:Yosida}\end{equation} satisfying (see \cite{Rock1998},
Exercise 12.23; for the definition of the subdifferential
$\partial\Psi^{\delta}$ see \eqref{eq:def-Subdiff})\begin{gather}
  \Psi^{\delta}\,:\,\symM\rightarrow\bbR\quad\mbox{is convex, coercive
    and continuously differentiable}\nonumber
  \\ \partial\Psi^{\delta}\,:\,\symM\rightarrow\symM\quad\mbox{is
    single valued and globally
    Lipschitz-continuous}\label{eq:prop-psi-delta}\\ \lim_{\delta\rightarrow0}\Psi^{\delta}(\sigma;\omega)=\Psi(\sigma;\omega)\qquad\forall\sigma\in\symM\,,\,\mbox{and
    a.e. }\omega\in\Omega\,.\nonumber \end{gather} Note that the last
convergence is monotone, since $\Psi^{\delta_2}\geq\Psi^{\delta_1}$
for all $\delta_2<\delta_1$. Given $\Psi$ and $\Psi^{\delta}$, we
consider the corresponding functionals
\begin{equation}
  \Upsilon, \Upsilon^{\delta}:\,\,\LOMs\rightarrow\Rr\,,\quad
  \Upsilon(z) := \int_{\Omega}\Psi(z(\omega))\, d\cP(\omega)\,,\
  \Upsilon^{\delta}(z) := \int_{\Omega}\Psi^{\delta}(z(\omega))\, d\cP(\omega)
  \,.\label{eq:Xi-Psi-duality-1}
\end{equation}
We denote by $\Upsilon_{n}:\, \LOMns\rightarrow\Rr$ the restriction of
$\Upsilon$ to $\LOMns$.  the subdifferential of $\Upsilon_{n}$ is
$\partial\Upsilon_{n}$. Accordingly, we can define
$\Upsilon_{n}^\delta$ and $\partial\Upsilon_{n}^\delta$.

\subsubsection*{The approximate problem for
  \eqref{eq:EPISIG-1}--\eqref{eq:EPISIG-2}}

We consider the following problem on discretized function spaces:
Given an evolution $\xi\in H^{1}_*(0,T;\symM)$, we look for
\begin{align*} 
  p_{\delta,n} & \in C^{1}(0,T;\LOMns)\,, & z_{\delta,n} & \in
  H^{1}(0,T;L_{sol,n}^{2}(\Omega))\,, & \v_{\delta,n} & \in
  H^{1}(0,T;\cV_{pot,n}^{2}(\Omega))\,,
\end{align*} 
with the symmetry $z_{\delta,n}=z_{\delta,n}^{s}$ , satisfying
\begin{equation}
  \partial_{t}p_{\delta,n}=\partial\Upsilon_{n}^{\delta}
  \left(z_{\delta,n}-B_{n}\, p_{\delta,n}\right)\label{eq:EPISIG-approx-1}
\end{equation}
and $C_{n}\, z_{\delta,n}= \xi + \v^s_{\delta,n} - p_{\delta,n}$. The
last equation can be written as
\begin{equation}
  z_{\delta,n}=C_{n}^{-1}\left(\xi+\v^s_{\delta,n}-p_{\delta,n}\right)\,.\label{eq:EPISIG-approx-2}
\end{equation}
Here, $B_{n},\, C_{n}:\LOMns\rightarrow\LOMns$ are bounded positive
(and thus invertible) operators defined through \[ \left\langle
  B_{n}\psi,\varphi\right\rangle
_{\Omega}=\int_{\Omega}\left(B\psi\right):\varphi\,,\quad\left\langle
  C_{n}\psi,\varphi\right\rangle
_{\Omega}=\int_{\Omega}\left(C\psi\right):\varphi\qquad\forall\varphi,\psi\in\LOMns\,.\]

We obtain the existence and uniqueness of solutions to
\eqref{eq:EPISIG-approx-1}--\eqref{eq:EPISIG-approx-2} from the
Picard-Lindel\"of theorem: We show that the system can be understood
as a single ordinary differential equation for $p_{\delta,n}$ with
Lipschitz continuous right hand side, and that the solutions are
uniformly bounded.

\subsection{\label{sub:Convex-functionals}Convex functionals }

\paragraph*{Basic concepts of convex functions.}

We recall some well known results from convex analysis on a separable
Hilbert space $X$ with scalar product {}``$\cdot$''. In the following,
$\varphi:X\to\Rr\cup\{+\infty\}$ is a convex and lower-semicontinuous
functional with $\varphi\not\equiv+\infty$.  The domain of $\varphi$
is $dom(\varphi):=\left\{ \sigma\in
  X|\varphi(\sigma)<+\infty\right\}$, and the Legendre-Fenchel
conjugate $\varphi^{*}$ is defined by \[
\varphi^{*}:X\to\Rr\cup\{+\infty\},\quad\varepsilon\mapsto\sup_{\sigma\in
  X}\{\varepsilon\cdot\sigma-\varphi(\sigma)\}.\] The subdifferential
$\partial\varphi:dom(\varphi)\to\mathcal{P}(X)$ is defined
by \begin{equation}
  \partial\varphi(\sigma)=\left\{ \varepsilon\in X\,|\,\varphi(\xi)\geq\varphi(\sigma)+\varepsilon\cdot(\xi-\sigma)\quad\forall\,\xi\in X\right\} .\label{eq:def-Subdiff}\end{equation}
A multivalued operator $f:dom(f)\subset X\to\mathcal{P}(X)$ is said to
be \textit{monotone} if \[
(\sigma_{1}-\sigma_{2})\cdot(\varepsilon_{1}-\varepsilon_{2})\geq0,\quad\forall\,\varepsilon_{i}\in
dom(f),\quad\sigma_{i}\in f(\varepsilon_{i}),\ (i=1,2).\] In what
follows, we frequently use the following properties of convex
functionals \cite{Rock1998}.
\begin{lem}
\label{lemma:convprop} For every convex and lower semicontinuous function $\varphi$ on a
Hilbert space $X$ with $\varphi\not\equiv+\infty$ holds 
\begin{itemize}
\item [(i)] $\varphi^{*}$ is convex, lower-semicontinuous, and
  $dom(\varphi^{*})\neq\emptyset$\vspace*{-2mm}
\item [(ii)] $\partial\varphi,\ \partial\varphi^{*}\ $ are monotone operators\vspace*{-2mm}
\item [(iii)]
  $\varphi(\sigma)+\varphi^{*}(\varepsilon)\geq\sigma\cdot\varepsilon
  \qquad\forall\,\sigma,\varepsilon\in X$\vspace*{-2mm}
\item [(iv)] $\sigma\in dom(\varphi)\mbox{ and
  }\varepsilon\in\partial\varphi(\sigma)\ \Leftrightarrow\
  \varepsilon\in dom(\varphi^{*})\mbox{ and
  }\sigma\in\partial\varphi^{*}(\varepsilon)$\vspace*{-2mm}
\item [(v)] $\varepsilon\in dom(\varphi^{*})\mbox{ and
  }\sigma\in\partial\varphi^{*}(\varepsilon)\ \ \Leftrightarrow\
  \varphi(\sigma)+\varphi^{*}(\varepsilon)=\sigma\cdot\varepsilon$\vspace*{-2mm}
\item [(vi)] $\varphi^{\ast\ast}=\varphi$.
\end{itemize}
\end{lem}
We refer to (v) as Fenchel's equality and to (iii) as Fenchel's \emph{in}equality.

\subsubsection*{Continuity properties of $\Upsilon$ and
  $\Upsilon^{\delta}$ and subdifferentials}

In order to obtain the subdifferential of the functional $\Upsilon :
\LOMs \to \R$ we calculate
\begin{align}
  &a\in\partial\Upsilon(z)\quad
  \Leftrightarrow\quad\Ups(z+\psi)\geq\Ups(z)
  +\scp a{\psi}_{\Omega}\quad\forall\psi\in\LOMs\nonumber \\
  &\qquad
  \Leftrightarrow\;\;\int_{\Omega}\Psi(z+\psi)\geq\int_{\Omega}\Psi(z)+\scp
  a{\psi}_{\Omega}\quad\forall\psi\in\LOMs\nonumber \\
  &\qquad \Leftrightarrow\;\; a(\omega)\in\partial\Psi(z(\omega))\text
  { for a.e. }\omega\in\Omega\,.\label{eq:def-subdif-Xi}
\end{align} 
Similarly, $a\in \partial\Ups^\delta(z)$ if and only if
$a\in\partial\Psi^\delta(z)$ almost everywhere. Both subdifferentials
are therefore single-valued and we may identify
$\partial\Upsilon^{\delta}(z)=\partial\Psi^{\delta}(z)$. We next
determine the subdifferential of the restricted functional
$\Upsilon_n^\delta$.
\begin{lem}
  The functionals $\Upsilon_{n}^{\delta}$ have a single valued
  subdifferential in every $z_{0}\in\LOMns$, given through
  \begin{equation}
    \partial\Upsilon_{n}^{\delta}(z_{0})=P_{n}\partial\Psi^{\delta}(z_{0})\,.
    \label{eq:funct-deriv}
  \end{equation}
\end{lem}

\begin{proof}
  Let $a\in\partial\Upsilon_{n}^{\delta}(z_{0})\subset\LOMns$ and let
  $id$ be the identity on $\LOMs$. For arbitrary $\varphi\in \LOMs$ we
  set $\varphi_{n}:=P_{n}\varphi$ and
  $\varphi_{o}:=(id-P_{n})\varphi$. We obtain
  \begin{align*}
    \int_{\Omega}\Psi^{\delta}(z_{0}+t\varphi)
    &=\Upsilon^{\delta}\left(z_{0}+t\varphi_{n}+t\varphi_{o}\right)
    \geq\Upsilon_{n}^{\delta}\left(z_{0}+t\varphi_{n}\right)
    +t\left\langle \partial\Psi^{\delta}\left(z_{0}+t\varphi_{n}\right),\varphi_{o}\right\rangle _{\Omega}\\
    & \geq\Upsilon_{n}^{\delta}\left(z_{0}\right)+t\left\langle
      a,\varphi_{n}\right\rangle
    _{\Omega}+t\left\langle \partial\Psi^{\delta}\left(z_{0}+t\varphi_{n}\right),\varphi_{o}\right\rangle
    _{\Omega}
  \end{align*}
  Since $\Psi^{\delta}$ is differentiable and $\partial\Psi^{\delta}$
  is Lipschitz continuous, we obtain from the fact that the
  subdifferential coincides with the derivative and from the last
  inequality
  \begin{align*}
    \left\langle \partial\Psi^{\delta}(z_{0}),\varphi\right\rangle
    _{\Omega} &
    =\lim_{t\rightarrow0}\frac{1}{t}\left(\int_{\Omega}\Psi^{\delta}(z_{0}+t\varphi)
      -\int_{\Omega}\Psi^{\delta}(z_{0})\right)\geq\left\langle
      a,\varphi_{n}\right\rangle
    _{\Omega}+\left\langle \partial\Psi^{\delta}\left(z_{0}\right),\varphi_{o}\right\rangle
    _{\Omega}\,.
  \end{align*}
  Replacing $\varphi$ by $-\varphi$ in the above calculations, we
  obtain
  $\partial\Psi^\delta(z_0)=a+(id-P_{n})\partial\Psi^{\delta}(z_{0})$
  or $P_{n} \partial\Psi^{\delta}(z_{0}) = a$.
\end{proof}

The Fenchel conjugate of $\Upsilon_{n}^{\delta}$ in $\LOMns$ is \[
\Upsilon_{n}^{\delta\ast}(\sigma):=\sup\left\{ \int_{\Omega}\sigma:e\,
  d\cP-\Upsilon_{n}^{\delta}(e)\,|\, e\in\LOMns\right\} \,.\] Since
$-\Upsilon_{n}^{\delta}(\cdot)$ is coercive in a finite dimensional
space, it has compact sublevels in $\LOMn$, and the supremum is indeed
attained.
\begin{lem}
\label{lem:3-3}\label{lem:weak-cont}Let $\Upsilon^{\delta\ast}$
be the Fenchel conjugate of $\Upsilon^{\delta}.$ For every $p\in\LOMs$
holds\begin{equation}
\Upsilon^{\ast}(p)=\int_{\Omega}\Psi^{\ast}(p)\, d\cP\,,\qquad
\Upsilon^{\delta\ast}(p)=\int_{\Omega}\Psi^{\delta\ast}(p)\, d\cP\,,\label{eq:Xi-Psi-duality-2}\end{equation}
and the functionals $\Upsilon$, $\Upsilon^{\ast}$, $\Upsilon^{\delta}$
and $\Upsilon^{\delta\ast}$ are convex and weakly lower semicontinuous
on $\LOMs$. 
\end{lem}

\begin{proof}
The functional $\Upsilon$ is convex with the conjugate
\[
\Upsilon^{\ast}(p):=\sup\left\{ \left\langle p,e\right\rangle
  _{\Omega}-\Upsilon(e)\,|\, e\in\LOMs\right\} \qquad\forall
p\in\LOMs\,.\] We first prove \eqref{eq:Xi-Psi-duality-2}: Let $p\in
dom\,\Upsilon^{\ast}=\LOMs$.  Since $\Ups^{\ast}$ is convex, we know
that $\partial\Upsilon^{\ast}(p)\not=\emptyset$.  Lemma
\ref{lemma:convprop} (iv) yields for any
$\sigma\in\partial\Upsilon^{\ast}(p)$ that $\sigma\in dom\,\Upsilon$
with $p\in\partial\Upsilon(\sigma)$ and Lemma \ref{lemma:convprop} (v)
then yields \begin{equation} \Upsilon^{\ast}(p)+\Upsilon(\sigma)=\scp
  p{\sigma}_{\Omega}\,.\label{eq:lem-3-3-help-1}\end{equation} Since
$p\in\partial\Upsilon(\sigma)$, \eqref{eq:def-subdif-Xi} yields
$p(\omega)\in\partial\Psi(\sigma(\omega);\omega)$ for
a.e. $\omega\in\Omega$ and Lemma \ref{lemma:convprop} (v) yields
$\Psi^{\ast}(p)+\Psi(\sigma)=p:\sigma$ a.e.. Integrating the last
equality over $\Omega$ and comparing with \eqref{eq:lem-3-3-help-1},
we find $\Upsilon^{\ast}(p)=\int_{\Omega}\Psi^{\ast}(p)$ since
$\Ups(\sigma)=\int_{\Omega}\Psi(\sigma)$.  The proof for the second
statement in \eqref{eq:Xi-Psi-duality-2} is similar.

We now prove the weak lower semicontinuity of $\Ups^{\ast}$. Let
$\sigma_{i}\in dom(\Psi)$, $i\in\Nn$, be dense in $dom(\Psi)$. We
define $\Psi_{m}^{\ast}$ as the maximum of finitely many functions\[
\Psi_{m}^{\ast}(p):=\max_{i=1,\dots,m}\left\{
  p:\sigma_{i}-\Psi(\sigma_{i})\right\} \qquad\forall p\in\symM\]
and note that $\Psi_m^\ast(p)\leq\Psi^\ast(p)$ for every $p\in\symM$.
For $z\in\LOMs$ and $i=1,\dots,m$, we introduce the sets \[ \Omega_{i}:=\left\{
  \omega\in\Omega\,|\,\Psi_{m}^{\ast}(z)=z:\sigma_{i}-\Psi(\sigma_{i})\right\}\backslash\bigcup_{j<i}\Omega_j
\,.\] 
Let $\left( z_{n}\right) _{n}$ be a sequence such
that $z_{n}\weakto z$ weakly in $\LOMs$. We find that 
\begin{multline*}
  \liminf_{n\rightarrow\infty}\int_{\Omega}\Psi^{\ast}(z_{n})\geq
  \liminf_{n\rightarrow\infty}\sum_{i=1}^m\int_{\Omega}\Psi^{\ast}_m(z_{n})
  =\liminf_{n\rightarrow\infty}\sum_{i}\int_{\Omega_{i}}\max_{j=1,\dots,m}\left(z_{n}:\sigma_{j}-\Psi(\sigma_{j})\right)\\
  \geq\liminf_{n\rightarrow\infty}\sum_{i}\int_{\Omega_{i}}\left(z_{n}:\sigma_{i}-\Psi(\sigma_{i})\right)
  =\sum_{i}\int_{\Omega_{i}}\left(z:\sigma_{i}-\Psi(\sigma_{i})\right)=\int_{\Omega}\Psi_{m}^{\ast}(z)\,.
\end{multline*}
Since $\Psi^{\ast}(p)=\lim_{m\rightarrow\infty}\Psi_{m}^{\ast}(p)$ for
every $p\in\symM$ by definition of $\Psi_{m}^{\ast}$, and since this
convergence is monotone, we can apply the monotone convergence theorem
and get
$\int_{\Omega}\Psi_{m}^{\ast}(z)\rightarrow\int_{\Omega}\Psi^{\ast}(z)
= \Ups^{\ast}(z)$.  This yields the weak lower semicontinuity of
$\Ups^{\ast}$.

Since $\Psi$ is convex and lower semicontinuous, we find $\Psi=\Psi^{\ast\ast}$
and switching $\Psi$ and $\Psi^{\ast}$ in the above argumentation,
the weak lower semicontinuity of $\Ups$ follows. The statements
for $\Ups^{\delta}$ and $\Ups^{\delta\ast}$ follow similarly. 
\end{proof}

\subsubsection*{Convergence properties}

We will later need additional lower semicontinuity properties: We have
to analyze the behavior of, e.g., $\Upsilon^{\delta}(u_{\delta})$.

\begin{lem}[Lower semicontinuity property of $\Psi^{\delta}$ and
  $\Psi^{\delta\ast}$]
  \label{lem:weak-lsc-psi}\label{lem:Ben-Veneroni-Lem-2.6:} Let
  $U_{s}:=\Omega\times(0,s)$ be the space-time cylinder and let
  $(u_{\delta})_\delta$ be a weakly convergent sequence,
  $u_{\delta}\rightharpoonup u$ weakly in $L^{2}(U_{s})$ as
  $\delta\rightarrow0$. Then, for $\Psi^{\delta}$, $\Psi$ as above, we
  find
  \begin{equation}
    \liminf_{\delta\rightarrow0}\int_{U_{s}}\Psi^{\delta\ast}(u_{\delta})\,
    d\cP\, dt\geq\int_{U_{s}}\Psi^{\ast}(u)\, d\cP\,
    dt\,.\label{eq:xi-lower-semicont-Sch-Ven}
  \end{equation} 
  For every sequence $(u_{\delta})_\delta$ with $u_{\delta}\weakto u$
  weakly in $\LOMs$ we find 
  \begin{equation}
    \liminf_{\delta\rightarrow0}\Ups^{\delta}(u_{\delta})\geq\Ups(u)\,.\label{eq:xi-lower-semicont}
  \end{equation}
\end{lem}

\begin{proof}
  The proof of \eqref {eq:xi-lower-semicont-Sch-Ven} is the same as in
  \cite{SchweizerVen10}, Lemma 2.6.

  Using the definition of $\Psi^{\delta}$ in \eqref{eq:Yosida}, we
  choose, for every $\delta>0$, a function $\pi_{\delta}\in\LOM$ such
  that
  \[
  \int_{\Omega}\left(\frac{\left|\pi_{\delta}-u_{\delta}\right|^{2}}{\delta}+\Psi(\pi_{\delta})\right)\,
  d\cP\leq\int_{\Omega}\Psi^{\delta}(u_{\delta})\, d\cP+\delta\,.\]
  Without loss of generality, we may assume
  $\liminf_{\delta\rightarrow0}\int_{\Omega}\Psi^{\delta}(u_{\delta})\,
  d\cP<\infty$. Then we get for a subsequence
  $\int_{\Omega}\left|\pi_{\delta}-u_{\delta}\right|^{2}\rightarrow0$
  as $\delta\rightarrow0$ and hence $\pi_{\delta}\weakto u$ weakly in
  $\LOM$ for this subsequence. Since
  $\int_{\Omega}\left|\pi_{\delta}-u_{\delta}\right|^{2}$ is positive
  and $\Ups(z)=\int_{\Omega}\Psi(z)$ is weakly lower semicontinuous,
  we find \eqref{eq:xi-lower-semicont}.
\end{proof}

The following lemma uses time-dependent functions and the
discretization parameter $n\in \Nn$.

\begin{lem}
  \label{lem:Ben-Veneroni-generalized}Let $s>0$ and let $p\in
  L^{2}(0,s;\LOMs)$ and $p_{n}\in L^{2}(0,s;\LOMn)$ such that
  $p_{n}\rightharpoonup p$ weakly in $L^{2}(0,s;\LOM)$ as
  $n\rightarrow\infty$. Then, for $\Upsilon_{n}^{\delta\ast}$ and
  $\Upsilon_{n}^{\delta}$ as above we find 
  \begin{equation}
    \liminf_{n\rightarrow\infty}\int_{0}^{s}\Upsilon_{n}^{\delta}(p_{n})\,
    dt\geq\int_{0}^{s}\Upsilon^{\delta}(p)\,
    dt\,,\qquad\liminf_{n\rightarrow\infty}\int_{0}^{s}\Upsilon_{n}^{\delta\ast}(p_{n})\,
    dt\geq\int_{0}^{s}\Upsilon^{\delta\ast}(p)\,
    dt\,.\label{eq:semi-cont-Xi}
  \end{equation} 
  Furthermore, if $z_{n}\to z$ strongly in $\LOM$ as
  $n\rightarrow\infty$, then 
  \begin{equation}
    \lim_{n\rightarrow\infty}\Upsilon_{n}^{\delta}(z_{n})=\Upsilon^{\delta}(z)\,.
    \label{eq:semi-cont-Xi-1}
  \end{equation}
\end{lem}

\begin{proof}
  Let $z_{n}\to z$ strongly in $\LOM$. Since $\Psi^{\delta}$ is
  Lipschitz continuous with $\Psi^{\delta}(0)=0$, we find because of
  $\Ups^\delta_n(z_n)=\Ups^\delta(z_n)$ \[
  \lim_{n\rightarrow\infty}\Upsilon_{n}^{\delta}(z_{n})
  =\lim_{n\rightarrow\infty}\int_{\Omega}\Psi^{\delta}(z_{n})=\int_{\Omega}\Psi^{\delta}(z)\]
  and thus \eqref{eq:semi-cont-Xi-1}. For $p_{n}\weakto p$ weakly in
  $L^{2}(U_{s})$ with $p_{n}\in L^{2}(0,s;\LOMn)$, the first
  inequality in \eqref{eq:semi-cont-Xi} can be proved similarly to the
  weak lower semicontinuity results of Lemma \ref{lem:weak-cont},
  using $\Upsilon_{n}^\delta(p_{n})=\Upsilon^\delta(p_{n})$.

  For the second inequality in \eqref{eq:semi-cont-Xi}, we choose
  finite sets $B_{n}=\left\{ e_{n}^{i}\,|\,i=1,\dots,K_{n}\right\}
  \subset\LOMn$ with $K_{n}\geq n$ such that $B_{n}\subset B_{n+1}$
  and $\bigcup_{n}B_{n}$ is dense in $\LOM$. For fixed $N\in\Nn$, the
  interval $[0,s]$ is split into subsets 
  \begin{equation}
    \tilde{\bbT}_{N}^{i}:=\left\{ t\in[0,s]\ |\ \max\left\{
        \left\langle e,p(t)\right\rangle
        _{\Omega}-\Upsilon^{\delta}(e)\,|\, e\in B_{N}\right\}
      =\left\langle e_{N}^{i},p(t)\right\rangle
      _{\Omega}-\Upsilon^{\delta}(e_{N}^{i})\right\} 
    \label{eq:splitting-T-Lemma-3-7}
  \end{equation}
  and we set $\bbT_N^1:=\tilde{\bbT}_N^1$ and
  $\bbT_N^i:=\tilde{\bbT}_N^i\backslash\bigcup_{j<i}\bbT_N^j$ for
  $i=2,\dots,K_N$.  For $n\geq N$ we find, decomposing the time
  integral, taking the maximum, performing the weak limit, and using
  the definition of $\bbT_N^i$:
  \begin{align*}
    \liminf_{n\rightarrow\infty}\int_{0}^{s}\Upsilon_{n}^{\delta\ast}(p_{n})
    &
    \geq\liminf_{n\rightarrow\infty}\sum_{i=1}^{K_{N}}\int_{\bbT_{N}^{i}}
    \max\left\{ \left\langle e,p_{n}(t)\right\rangle _{\Omega}
      -\Upsilon_{n}^{\delta}(e)\,|\, e\in B_{N}\right\} dt\\
    & \geq\liminf_{n\rightarrow\infty}\sum_{i=1}^{K_{N}}
    \int_{\bbT_{N}^{i}}\left(\left\langle
        e_{N}^{i},p_{n}(t)\right\rangle _{\Omega}
      -\Upsilon_{n}^{\delta}(e_{N}^{i})\right)dt\\
    & =\sum_{i=1}^{K_{N}}\int_{\bbT_{N}^{i}}
    \left(\left\langle e_{N}^{i},p(t)\right\rangle _{\Omega}-\Upsilon^{\delta}(e_{N}^{i})\right)dt\\
    & \stackrel{{\scriptstyle
        \eqref{eq:splitting-T-Lemma-3-7}}}{=}\sum_{i=1}^{K_{N}}\int_{\bbT_{N}^{i}}
    \max\left\{ \left\langle e,p(t)\right\rangle
      _{\Omega}-\Upsilon^{\delta}(e)\,|\, e\in
      B_{N}\right\} \, dt\\
    &= \sup\left\{ \int_{0}^{s}\left(\left\langle
          \tilde{e},p(t)\right\rangle
        _{\Omega}-\Upsilon^{\delta}(\tilde{e}(t))\right)dt\,|\,\tilde{e}\in
      L^{2}(0,s;B_{N})\right\} \,.
  \end{align*}
  This inequality implies, due to density of $\bigcup_{N} B_{N}$ in $\LOM$,
  \begin{align*}
    \liminf_{n\rightarrow\infty}\int_{0}^{s}\Upsilon_{n}^{\delta\ast}(p_{n})
    &\geq\sup\left\{ \int_{0}^{s}\int_{\Omega}\left(e:p-\Psi^{\delta}(e)\right)\,|\, 
      e\in L^{2}(0,s;\LOM)\right\} \\
    &=\int_{0}^{s}\int_{\Omega}\Psi^{\delta\ast}(p)=\int_0^s\Upsilon^{\delta\ast}(p)\,,
  \end{align*}
  where we used \eqref{eq:Xi-Psi-duality-2} in the last equality. We
  have thus verified the second inequality of \eqref{eq:semi-cont-Xi}.
\end{proof}

\subsection{Properties of $\cV^2_{pot}(\Omega)$-functions}
\label{sub:Korn-inequality}

\begin{lem}[Potentials with small norm]
  \label{lem:existence-primitive-v-omega}Let $U\subset\Rn$ be a
  bounded Lipschitz-domain and let $\v\in\cV_{pot}^{2}(\Omega)$. Then,
  for $\cP$-a.e.  $\omega\in\Omega$ and every $\eps>0$ there exists
  $\phi_{\eps,\omega,\v}\in H^{1}(U;\Rn)$ such that
  $\nabla\phi_{\eps,\omega, \v}(x)=\v(\tau_{\frac{x}{\eps}}\omega)$
  and such that
  \[
  \lim_{\eps\to0}\norm{\phi_{\eps,\omega,\v}}_{L^{2}(U)} = 0\,.
  \]
\end{lem}

\begin{proof}
  Let $\v\in\cV_{pot}^{2}(\Omega)$ and write
  $\v_{\eps,\omega}(x):=\v(\tau_\frac{x}{\eps} \omega)$.  By the
  ergodic theorem \ref{thm:ergodic-thm}, there exists
  $\Omega_\v\subset\Omega$ with $\cP(\Omega_\v)=1$ such that for all
  $\omega\in\Omega_\v$ there exists $C_\omega>0$ with
  \begin{equation}\label{eq:limit-step-functions-help-a}
    \sup_{\eps>0}\|\v_{\eps,\omega}\|_{L^2(U)}\leq C_\omega\,.
  \end{equation}
  Let $\left(\varphi_i\right)_{i\in\Nn}\subset L^2(\Omega;\Rdd)$ a
  countably dense family. For every $i\in\Nn$ there exists
  $\Omega_i\subset\Omega$ with $\cP(\Omega_i)=1$ such that for every
  $\omega\in\Omega_i$
  \begin{equation}\label{eq:limit-step-functions-help}
    \int_U \v_{\eps,\omega}(x)\varphi_i(x)\,dx\to\int_U
    \left(\int_\Omega \v\,d\cP\right)\varphi_i \,dx=0\qquad\text{as }\eps\to0\,.
  \end{equation}
  We define $\tilde\Omega:=\Omega_\v\cup\bigcup_{i\in\Nn}\Omega_i$.
  By \eqref{eq:limit-step-functions-help-a} and
  \eqref{eq:limit-step-functions-help} we obtain that
  $\v_{\eps,\omega}(x)\weakto0$ as $\eps\to0$ for all
  $\omega\in\tilde\Omega$.

  By the definition of $L^2_{pot}(\Omega)$ in \eqref{def:L2pot}, there
  exists $\phi_{\eps,\omega,\v}\in H^{1}(U)$ such that
  $\nabla\phi_{\eps,\omega, \v}(x)=\v(\tau_{\frac{x}{\eps}}\omega)$.
  By adding a constant, we can achieve
  $\int_{U}\phi_{\eps,\omega,\v}=0$.  By the Poincar\'e inequality, it
  follows that
\[
\norm{\phi_{\eps,\omega,\v}}_{L^{2}(U)}\leq\norm{\nabla\phi_{\eps,\omega,
    \v}(x)}_{L^{2}(U)}+\left|\int_{U}\phi_{\eps,\omega,\v}\right|=\norm{\nabla\phi_{\eps,\omega,
    \v}(x)}_{L^{2}(U)}=\int_U |\v_{\eps,\omega}(x)|^2\,.
\] 
Since the family $\phi_{\eps,\omega,\v}$ is bounded in $H^1(U)$, it is
precompact in $L^2(U)$.  We chose $f\in C^\infty_{c}(U;\Rn)$ and
denote by $F$ the solution to the Neumann boundary problem $-\Delta
F=f$. We obtain
\begin{equation*}
  -\lim_{\eps\to0}\int_{U}\phi_{\eps,\omega,\v}\cdot f  =\lim_{\eps\to0}\int_{U}\nabla\phi_{\eps,\omega,\v}:\nabla F=\lim_{\eps\to0}\int_{U}\v(\tau_{\frac{x}{\eps}}\omega):\nabla F(x)\,dx=0\,.
\end{equation*}
Therefore, $\phi_{\eps,\omega,\v}\weakto0$ in $L^{2}(U)$. Since
$\left(\phi_{\eps,\omega,\v}\right)_{\eps>0}$ is precompact in
$L^{2}(U)$, it follows that $\phi_{\eps,\omega,\v}\to0$ in $L^{2}(U)$.
\end{proof}

\begin{lem}[A Korn's inequality on $\Omega$]
  \label{lem:Korn-Omega} For every $f\in \cV_{pot}^{2}(\Omega)$ holds
  \begin{equation}
    \norm{f}_{\LOM}\leq2\norm{f^{s}}_{\LOM}\,.\label{eq:Korn-Omega}
  \end{equation}
\end{lem}

\begin{proof}
  In what follows, we denote $Q:=(-1,1)^d$ and
  $Q_\eta:=(-1+\eta,1-\eta)^d$ for $\frac12>\eta>0$.  We choose
  $\psi_\eta\in C^\infty_c(Q)$ with $0\leq\psi_\eta\leq1$,
  $\psi_\eta\equiv1$ on $Q_\eta$ and $|\nabla\psi_\eta|<2\eta^{-1}$.
	
  Let $f\in \cV^{2}_{pot}(\Omega)$ and for every $\eps>0$ and
  $\omega\in\Omega$ let $\phi_{\eps,\omega,f}$ denote the potential of
  $f_\omega$ from Lemma \ref{lem:existence-primitive-v-omega}.  If we
  denote the characteristic function of $Q\backslash Q_\eta$ by
  $\chi_{Q\backslash Q_\eta}$, we have the pointwise inequality
  $$\left|\,\left|\nabla\phi_{\eps,\omega,f}\right|^2 
    - \left|\nabla\left(\phi_{\eps,\omega,f}\psi_\eta\right)\right|^2
  \right| \leq \chi_{Q\backslash Q_\eta} \left(
    2\left|\nabla\phi_{\eps,\omega,f}\right|^2 + \frac4\eta
    \left|\phi_{\eps,\omega,f}\right|\,\left|\nabla\phi_{\eps,\omega,f}\right|
    + \frac{4}{\eta^2}\left|\phi_{\eps,\omega,f}\right|^2\right)
  $$
  Using this inequality, we get from the ergodic theorem
  \ref{thm:ergodic-thm} and Lemma
  \ref{lem:existence-primitive-v-omega} for
  $\cP$-a.e. $\omega\in\Omega$
  \begin{align}
    &\lim_{\eps\to0}\int_{Q}\left|\,\left|\nabla\phi_{\eps,\omega,f}\right|^2
      - \left|\nabla\left(\phi_{\eps,\omega,f}\psi_\eta\right)\right|^2 \right|\nonumber\\
    &\qquad\qquad \leq \lim_{\eps\to0}\int_{Q\backslash
      Q_\eta}2\left|\nabla\phi_{\eps,\omega,f}\right|^2
    +\frac{4}{\eta}\lim_{\eps\to0}\|\phi_{\eps,\omega,f}\|_{L^2(Q)}\|\nabla\phi_{\eps,\omega,f}\|_{L^2(Q)}
    +\frac{4}{\eta^2}\lim_{\eps\to0}\|\phi_{\eps,\omega,f}\|_{L^2(Q)}^2\nonumber\\
    &\qquad\qquad =2|Q\backslash Q_\eta|\int_\Omega
    f^2\,d\cP\,,\label{eq:proof-korn1}
  \end{align}
  where we have used that $\phi_{\eps,\omega,f}\to0$ strongly in
  $L^2(Q)$.  Arguing along the same limes with symmetrized functions,
  we can show that
  \begin{equation}
    \lim_{\eps\to0}\int_{Q}\left|\,\left|\nabla^s\phi_{\eps,\omega,f}\right|^2 
      - \left|\nabla^s\left(\phi_{\eps,\omega,f}\psi_\eta\right)\right|^2 \right|
    \leq 2|Q\backslash Q_\eta|\int_\Omega (f^s)^2\,d\cP\,,\label{eq:proof-korn2}
  \end{equation}
  Since $(\phi_{\eps,\omega,f}\psi_\eta)\in H^1_0(Q)$, we can apply
  Korn's inequality in $\Rn$ and obtain
  \begin{equation}
    \int_Q\left|\nabla\left(\phi_{\eps,\omega,f}\psi_\eta\right)\right|^2\leq 
    2 \int_Q\left|\nabla^s\left(\phi_{\eps,\omega,f}\psi_\eta\right)\right|^2\,.\label{eq:proof-korn3}
  \end{equation}
  Combining \eqref{eq:proof-korn1}--\eqref{eq:proof-korn3} with the
  ergodic theorem \ref{thm:ergodic-thm}, we obtain that
  \begin{align*}
    & |Q|\,\int_{\Omega}\left|f\right|^{2}d\cP\stackrel{{\scriptstyle \ref{thm:ergodic-thm}}}{=}\lim_{\eps\to0}\int_{Q}\left(f(\tau_{\frac x\eps}\omega)\right)^2\, dx=\lim_{\eps\to0}\int_{Q}\left|\nabla\phi_{\eps,\omega,f}\right|^2\, dx\\
    &\qquad\stackrel{{\scriptstyle \eqref{eq:proof-korn1}}}{\leq} \lim_{\eps\to0}\int_Q\left|\nabla\left(\phi_{\eps,\omega,f}\psi_\eta\right)\right|^2+ 2|Q\backslash Q_\eta|\int_\Omega f^2\,d\cP\displaybreak[0]\\
    &\qquad\stackrel{{\scriptstyle \eqref{eq:proof-korn3}}}{\leq} \lim_{\eps\to0}2\int_Q\left|\nabla^s\left(\phi_{\eps,\omega,f}\psi_\eta\right)\right|^2+ 2|Q\backslash Q_\eta|\int_\Omega f^2\,d\cP\displaybreak[0]\\
    &\qquad\stackrel{{\scriptstyle \eqref{eq:proof-korn2}}}{\leq} \lim_{\eps\to0}2\int_Q\left|\nabla^s\phi_{\eps,\omega,f}\right|^2+ (2+4)|Q\backslash Q_\eta|\int_\Omega f^2\,d\cP\\
    &\qquad\stackrel{{\scriptstyle \ref{thm:ergodic-thm}}}{\leq}
    |Q|\,2\int_\Omega\left|f^s\right|^2\,d\cP+ 6|Q\backslash
    Q_\eta|\int_\Omega f^2\,d\cP\,.
  \end{align*}  
  Since the last estimate holds for every small $\eta>0$, we obtain
  inequality \eqref{eq:Korn-Omega}.
\end{proof}

\subsection{\label{sub:Approximation-procedure}Solutions to the
  approximate problem and a priori estimates}

\begin{lem}
  \label{lem:EPISIG-apriori-1} There exists a unique solution
  $p_{\delta,n}$, $z_{\delta,n}$, $\v_{\delta,n}$ to problem
  \eqref{eq:EPISIG-approx-1}--\eqref{eq:EPISIG-approx-2} which
  satisfies the a priori estimate\begin{equation}
    \norm{p_{\delta,n}}_{\cV^{1}_{0}}+\norm{z_{\delta,n}}_{\cV^{1}_{0}}+\norm{\v_{\delta,n}}_{\cV^{1}_{0}}\leq
    c\left(\Ups_{n}^{\delta}(z_{\delta,n}(0))+\norm{\xi}_{H^{1}(0,T)}\right)\,,\label{eq:lem-3-2-1}
  \end{equation}
  with $\cV^{1}_{0}:=H^{1}(0,T;L^{2}(\Omega;\symM))$ and $c$
  independent of $\delta$ and $n$.
\end{lem}

\begin{proof}
  In the following, all integrals over $\Omega$ are with respect to
  $\cP$ and we omit $d\cP$ for ease of notation. We will prove the
  lemma in two steps: we first show that the system
  \eqref{eq:EPISIG-approx-1}--\eqref{eq:EPISIG-approx-2} is equivalent
  to an ordinary differential equation for $p_{\delta,n}$ with
  Lipschitz continuous right hand side. Then, we show that the
  solution admits uniform a priori estimates.

  \smallskip \textbf{Step 1: Existence.} In order to study
  \eqref{eq:EPISIG-approx-1}--\eqref{eq:EPISIG-approx-2}, we fix
  $\tilde{p}\in\LOMns$ and $\tilde{\xi}\in\symM$, and search for
  $\tilde{\v}\in \cV_{pot,n}^{2}(\Omega)$ such that 
  \begin{equation}
    \scp{C_{n}^{-1}\tilde{\v}^{s}}{\zeta}_{\Omega}
    =\scp{C_{n}^{-1}\tilde{p}}{\zeta}_{\Omega}
    -\scp{C_{n}^{-1}\tilde{\xi}}{\zeta}_{\Omega}\qquad\forall\zeta\in
    \cV_{pot,n}^{2}(\Omega)\,.\label{eq:z-is-sol}
  \end{equation} 
  The Lax-Milgram theorem in combination with Korn's inequality
  \eqref{eq:Korn-Omega} yields a unique solution $\tilde{\v}\in
  \cV_{pot,n}^{2}(\Omega)$ of the last equality. We introduce the
  mapping $V_{\tilde{\xi}}:\,\LOMns\rightarrow \cV_{pot,n}^{2}(\Omega)$
  with $V_{\tilde{\xi}}(\tilde{p})=\tilde{\v}$ and note that this
  operator is linear and bounded. We then look for a solution
  $p_{\delta,n}\in C^{1}(0,T;\LOMn)$ to the following version of
  \eqref{eq:EPISIG-approx-1}: \[
  \partial_{t}p_{\delta,n}=\partial\Upsilon_{n}^{\delta}
  \left(C_{n}^{-1}\left(\xi+V_{\xi}(p_{\delta,n})^s -
      p_{\delta,n}\right)-B_{n}\, p_{\delta,n}\right)\,.\] Relation
  \eqref{eq:funct-deriv} yields the Lipschitz continuity of
  $\partial\Upsilon_{n}^{\delta}$. Therefore, since also
  $\partial\Upsilon_{n}^{\delta}$, $C_{n}^{-1}$, $V_{\xi}^{s}$ and
  $B_{n}$ are Lipschitz-continuous mappings $\LOMns\rightarrow\LOMns$,
  we find a unique solution $p_{\delta,n}\in C^{1}([0,T];\LOMns)$ of
  the ordinary differential equation (a priori bounds are provided
  below). We furthermore set $\v_{\delta,n}=V_{\xi}(p_{\delta,n})\in
  C^{1}([0,T];\cV_{pot,n}^{2}(\Omega))$ and
  $z_{\delta,n}=C_{n}^{-1}\left(\xi+\v_{\delta,n}^{s}-p_{\delta,n}\right)\in
  H^{1}(0,T;\LOMns)$.  From \eqref{eq:z-is-sol} and the definition of
  $\v_{\delta,n}$, it follows that $z_{\delta,n}\in
  H^{1}(0,T;L_{sol,n}^{2}(\Omega))$.  Note that $p_{\delta,n}$,
  $z_{\delta,n}$ and $\v_{\delta,n}$ are constructed in such a way
  that \eqref{eq:EPISIG-approx-1}--\eqref{eq:EPISIG-approx-2}
  holds. The construction shows that the solution is uniquely
  determined.

  \smallskip \textbf{Step 2: A priori estimates of order 0.} We take
  the time derivative of \eqref{eq:EPISIG-approx-2}, multiply by
  $z_{\delta,n}$ and integrate over $[0,t]\times\Omega$ for
  $t\in(0,T]$ to find 
  \begin{align}
    \int_{0}^{t}\int_{\Omega} \partial_{t}\xi : z_{\delta,n} &
    \stackrel{{\scriptstyle \eqref{eq:EPISIG-approx-2}}}{=}\int_{0}^{t}\int_{\Omega}\left(\left(C_{n}\partial_{t}z_{\delta,n}\right):z_{\delta,n}+\partial_{t}p_{\delta,n}:z_{\delta,n}-\partial_{t}\v^s_{\delta,n}:z_{\delta,n}\right)\nonumber \\
    & =\frac{1}{2}\left.\int_{\Omega}\left(p_{\delta,n}:B_{n}p_{\delta,n} +z_{\delta,n}:C_nz_{\delta,n}\right)\right|_{0}^{t}\nonumber \\
    & \phantom{=}+\int_{0}^{t}\scp{\partial_{t}p_{\delta,n}}{z_{\delta,n}-B_{n}p_{\delta,n}}_{\Omega}-\int_{0}^{t}\int_{\Omega}z_{\delta,n}:\partial_{t}\v_{\delta,n}\nonumber \\
    & \stackrel{{\scriptstyle (\ast)}}{=}\frac{1}{2}\left.\int_{\Omega}\left(p_{\delta,n}:(Bp_{\delta,n})+z_{\delta,n}:(Cz_{\delta,n})\right)\right|_{0}^{t}\nonumber \\
    &
    \phantom{=}+\int_{0}^{t}\left(\Upsilon_{n}^{\delta\ast}\left(\partial_{t}p_{\delta,n}\right)+\Upsilon_{n}^{\delta}\left(z_{\delta,n}-B_{n}p_{\delta,n}\right)\right)\,.\label{eq:EPISIG-est-1-a}
  \end{align}
  In $(\ast)$ we used the orthogonality of potentials and (symmetric)
  solenoidals,
  $\int_{\Omega}z_{\delta,n}:\partial_{t}\v_{\delta,n}=0$, and Lemma
  \ref{lemma:convprop} (v), written as \[ \scp{\del_t
    p}{z-Bp}_{\Omega}=\Upsilon_{n}^{\delta}(z-Bp)
  +\Upsilon_{n}^{\delta\ast}(\del_t p)\quad\Leftrightarrow\quad \del_t
  p=\partial\Upsilon_{n}^{\delta}(z-Bp)\,.\] \textbf{A priori
    estimates of order 1.} Taking the time derivative of
  \eqref{eq:EPISIG-approx-2}, multiplying the result by
  $\partial_{t}z_{\delta,n}$ and integrating over $\Omega$, we
  get \begin{align*}
    \int_{\Omega} \partial_{t}\xi : \partial_{t}z_{\delta,n} & =\int_{\Omega}\partial_{t}z_{\delta,n}:\partial_{t}\left(p_{\delta,n}+C_{n}z_{\delta,n}-\v_{\delta,n}\right)+\int_{\Omega}\left(B_{n}\partial_{t}p_{\delta,n}-B_{n}\partial_{t}p_{\delta,n}\right):\partial_{t}p_{\delta,n}\displaybreak[2]\\
    & \stackrel{{\scriptstyle \eqref{eq:EPISIG-approx-1}}}{=}\scp{\partial_{t}z_{\delta,n}-B_{n}\partial_{t}p_{\delta,n}}{\partial\Upsilon_{n}^{\delta}(z_{\delta,n}-B_{n}p_{\delta,n})}_{\Omega}+\int_{\Omega}\left(B_{n}\partial_{t}p_{\delta,n}\right):\partial_{t}p_{\delta,n}\\
    & \phantom{=}+\int_{\Omega}\left(C_{n}\partial_{t}z_{\delta,n}\right):\partial_{t}z_{\delta,n}-\int_{\Omega}\partial_{t}z_{\delta,n}:\partial_{t}\v_{\delta,n}\displaybreak[2]\\
    & \stackrel{{\scriptstyle
        (\ast)}}{=}\frac{d}{dt}\Upsilon_{n}^{\delta}(z_{\delta,n}-B_{n}p_{\delta,n})+\int_{\Omega}\left(\left(C\partial_{t}z_{\delta,n}\right):\partial_{t}z_{\delta,n}+\left(B\partial_{t}p_{\delta,n}\right):\partial_{t}p_{\delta,n}\right)\,,\end{align*}
  where we used
  $\int_{\Omega}\partial_{t}z_{\delta,n}:\partial_{t}\v_{\delta,n}=0$
  in $(\ast)$. We integrate the last equality over $(0,t)$ for
  $t\in(0,T]$ and obtain\begin{multline}
    \Upsilon_{n}^{\delta}(z_{\delta,n}(0))+\int_{0}^{t}\int_{\Omega}\partial_{t}z_{\delta,n}:\partial_{t}\xi^{s}\\
    \geq\Upsilon_{n}^{\delta}(z_{\delta,n}(t)-B_{n}p_{\delta,n}(t))+\int_{0}^{t}\int_{\Omega}\left(\left(C\partial_{t}z_{\delta,n}\right):\partial_{t}z_{\delta,n}+\left(B\partial_{t}p_{\delta,n}\right):\partial_{t}p_{\delta,n}\right)\,.\label{eq:EPISIG-ineq-chain-1}\end{multline}
  Since $\Upsilon_{n}^{\delta\ast}$ and $\Upsilon_{n}^{\delta}$ are
  positive, we can neglect them in \eqref{eq:EPISIG-est-1-a}. Applying
  the Cauchy-Schwarz inequality to the right hand side of
  \eqref{eq:EPISIG-est-1-a} and then Gronwall's inequality yields an
  estimate\[
  \sup_{t\in[0,T]}\norm{z_{\delta,n}(t)}_{\LOM}+\sup_{t\in[0,T]}\norm{p_{\delta,n}(t)}_{\LOM}\leq
  c\norm{\xi}_{H^{1}}\,.\] From positivity of $\Upsilon_{n}^{\delta}$
  on the right hand side of \eqref{eq:EPISIG-ineq-chain-1}, it follows
  that\[
  \int_{0}^{t}\int_{\Omega}\left(\left(C\partial_{t}z_{\delta,n}\right):\partial_{t}z_{\delta,n}+\left(B\partial_{t}p_{\delta,n}\right):\partial_{t}p_{\delta,n}\right)\leq\Upsilon_{n}^{\delta}(z_{\delta,n}(0))+\norm{\xi}_{H^{1}}\,.\]
  The last two inequalities yield \eqref{eq:lem-3-2-1} for
  $z_{\delta,n}$ and $p_{\delta,n}$. The inequality for
  $\v_{\delta,n}$ follows from equation \eqref{eq:EPISIG-approx-2}.
\end{proof}

\subsection{Proof of Theorem \ref{thm:E-Pi-Sig-well-def}}
\label{sub:Proof-of-Thm-Main2}

\paragraph*{Existence.}

Using the sequence $(p_{\delta,n},z_{\delta,n},\v_{\delta,n})$ of
solutions to \eqref{eq:EPISIG-approx-1}--\eqref{eq:EPISIG-approx-2},
we can now prove Theorem \ref{thm:E-Pi-Sig-well-def}. For
$n\rightarrow\infty$, we find weakly convergent subsequences of
$p_{\delta,n}$, $z_{\delta,n}$, $\v_{\delta,n}$ in $\cV^{1}_{0}$ with
limits $p_{\delta}$, $z_{\delta}$, $\v_{\delta}$.  We note that
$z_{\delta,n}(0)$ is the unique solution in $L_{sol,n}^{2}(\Omega)$
to \[
\int_{\Omega}\left(C_{n}z_{\delta,n}(0)\right):\psi=\int_{\Omega}\xi(0):\psi\qquad\forall\psi\in
L_{sol,n}^{2}(\Omega)\,.\] Hence, since we consider only $\xi$ with
$\xi(0) = 0$, the initial values $z_{\delta,n}(0)$ vanish identically.
As a consequence, also $\Ups_n^\delta(z_{\delta,n}(0))$ in
\eqref{eq:lem-3-2-1} vanishes.  The estimate \eqref{eq:lem-3-2-1}
therefore implies \eqref{eq:EPISIG-apriori-global} for
$(p_{\delta},z_{\delta},\v_{\delta})$.

Since $p_{\delta,n}$, $z_{\delta,n}$, $\v_{\delta,n}$ satisfy
\eqref{eq:EPISIG-approx-2}, the limits $p_{\delta}$, $z_{\delta}$,
$\v_{\delta}$ satisfy \begin{equation} C
  z_{\delta}=\xi+\v_{\delta}^s-p_{\delta}\,.\label{eq:EPISIG-approx-2-delta}\end{equation}
We take the limit $n\rightarrow\infty$ in \eqref{eq:EPISIG-est-1-a},
apply Lemma \ref{lem:Ben-Veneroni-generalized} and exploit the
vanishing initial data to conclude that the functions $p_{\delta}$,
$z_{\delta}$, $\v_{\delta}$ satisfy \begin{equation}
  \int_{0}^{t}\int_{\Omega}\left(\Psi^{\delta\ast}\left(\partial_{t}p_{\delta}\right)+\Psi^{\delta}\left(z_{\delta}-B\,
      p_{\delta}\right)\right)\leq\int_{0}^{t}\int_{\Omega}z_{\delta}:\partial_{t}\xi-\frac{1}{2}\left.\int_{\Omega}\left(p_{\delta}:(B\,
      p_{\delta})+z_{\delta}:(Cz_{\delta})\right)\right|_{0}^{t}\,.\label{eq:EPISIG-ineq-chain}
\end{equation}

In the limit $\delta\rightarrow0$ we find weakly convergent
subsequences of $p_{\delta}$, $z_{\delta}$, $\v_{\delta}$ with the
respective weak limits $p$, $z$, $\v$ satisfying the estimate
\eqref{eq:EPISIG-apriori-global}.  Passing to the limit
$\delta\rightarrow0$ in \eqref{eq:EPISIG-approx-2-delta}, we find that
$(p,z,\v)$ satisfies \eqref{eq:EPISIG-2}. Furthermore, passing to the
limit in \eqref{eq:EPISIG-ineq-chain}, using Lemma
\ref{lem:weak-lsc-psi}, we find that the functions $p$, $z$, $\v$
satisfy \[
\int_{0}^{t}\int_{\Omega}\left(\Psi^{\ast}\left(\partial_{t}p\right)+\Psi\left(z-B\,
    p\right)\right)\leq\int_{0}^{t}\int_{\Omega}z:\partial_{t}\xi-\frac{1}{2}\left.\left(\int_{\Omega}p:(Bp)+\int_{\Omega}z:(Cz)\right)\right|_{0}^{t}\,.\]
We thus obtain
\begin{align*} &
  \int_{0}^{t}\int_{\Omega}\left(\Psi^{\ast}\left(\partial_{t}p\right)+\Psi\left(z-B\,
  p\right)\right)=\int_{0}^{t}\int_{\Omega}\left(z:\partial_{t}\xi-\partial_{t}p:Bp-\partial_{t}z:Cz\right)\\ &
  \qquad\qquad\stackrel{{\scriptstyle
      \eqref{eq:EPISIG-2}}}{=}\int_{0}^{t}\int_{\Omega}\left(z:C\partial_{t}z-
  z:\partial_{t}\v^s+
  z:\partial_{t}p-\partial_{t}p:Bp-\partial_{t}z:Cz\right)\\ &
  \qquad\qquad\stackrel{\phantom{{\scriptstyle
        \eqref{eq:EPISIG-2}}}}{=}\int_{0}^{t}\int_{\Omega}\left(-z:\partial_{t}\v^s+\partial_{t}p:\left(z-B\,
  p\right)\right) \stackrel{\phantom{{\scriptstyle
        \eqref{eq:EPISIG-2}}}}{=}\int_{0}^{t}\int_{\Omega}\partial_{t}p:\left(z-B\,
  p\right)\,\end{align*} for every $t\in(0,T)$. On the other hand,
  since Lemma \ref{lemma:convprop} (iii) yields
  $\left(\Psi^{\ast}\left(\partial_{t}p\right)+\Psi\left(z-Bp\right)\right)\geq\partial_{t}p:\left(z-Bp\right)$
  pointwise a.e., we find \[
  \left(\Psi^{\ast}\left(\partial_{t}p\right)+\Psi\left(z-Bp\right)\right)=\partial_{t}p:\left(z-Bp\right)\]
  pointwise a.e. in $(0,T)\times\Omega$. The Fenchel equality of Lemma
  \ref{lemma:convprop} (v) then yields \eqref{eq:EPISIG-1}.

\paragraph*{Uniqueness and continuity.}

Let $\xi_{1},\xi_{2}\in H^{1}_*(0,T;\symM)$. Let
$\left(p_{i},z_{i},\v_{i}\right)_{i\in\left\{ 1,2\right\} }$ be two
solutions to \eqref{eq:EPISIG-1}-\eqref{eq:EPISIG-2} for $\xi_{1}$,
$\xi_{2}$ respectively with the difference
$\left(\tilde{p},\tilde{z},\tilde{\v}\right):=\left(p_{1},z_{1},\v_{1}\right)-\left(p_{2},z_{2},\v_{2}\right)$.
We integrate $\tilde{z}:\partial_{t}\left(\xi_{1}-\xi_{2}\right)$ over
$\Omega$ and obtain from a calculation similar to
\eqref{eq:EPISIG-est-1-a}
\begin{align*}
  &
  \left.\int_{\Omega}\tilde{z}:(\xi_{1}-\xi_{2})\right|_{0}^{t}-\int_{0}^{t}\int_{\Omega}\partial_{t}\tilde{z}:(\xi_{1}-\xi_{2})\\
  &
  \qquad\qquad=\int_{0}^{t}\int_{\Omega}\tilde{z}:\partial_{t}(\xi_{1}-\xi_{2})=\int_{\Omega}\tilde{z}:\partial_{t}\left(C\tilde{z}-\tilde{p}+\tilde{\v}\right)\displaybreak[2]\\
  &
  \qquad\qquad=\frac{1}{2}\frac{d}{dt}\int_{\Omega}\left(\tilde{p}:(B\tilde{p})+\tilde{z}:(C\tilde{z})\right)\\
  &
  \qquad\qquad\phantom{=}+\int_{\Omega}\left[\left(z_{1}(t,\omega)-B(\omega)\,
      p_{1}(t,\omega)\right)-\left(z_{2}(t,\omega)-B(\omega)\,
      p_{2}(t,\omega)\right)\right]\left(\partial_{t}p_{1}-\partial_{t}p_{2}\right)\,.
\end{align*}
From the monotonicity of $\partial\Psi$ (Lemma \ref{lemma:convprop}
(ii)) and \eqref{eq:EPISIG-1}$_{1,2}$, we find \[
\frac{1}{2}\left.\int_{\Omega}\left(\tilde{p}:(B\tilde{p})+\tilde{z}:(C\tilde{z})\right)\right|_{0}^{t}
\le
\left.\int_{\Omega}\tilde{z}:(\xi_{1}-\xi_{2})\right|_{0}^{t}-\int_{0}^{t}\int_{\Omega}\partial_{t}\tilde{z}:(\xi_{1}-\xi_{2})
\]
for every $t\in(0,T)$.  Compactness of the embedding $H^{1}(0,T;\symM)
\subset C([0,T];\symM)$ and boundedness of $\partial_{t}\tilde{z}$
provide the weak continuity of the mapping $\xi\mapsto(z,p,\v)$.  At
the same time, it implies uniqueness of solutions,
i.e. $\left(\tilde{p},\tilde{z},\tilde{\v}\right)=(0,0,0)$ for
$\xi_{1}=\xi_{2}$. This completes the proof of Theorem
\ref{thm:E-Pi-Sig-well-def}.

\section{Proof of the main theorem}
\label{sec.proof-main}

\subsection{Preliminaries}

\begin{lem}[A time dependent ergodic theorem]
  \label{lem:ergodic-time}Let $f\in L^{p}(0,T;L^{p}(\Omega))$, $1\leq
  p<\infty$ and $f_{\omega}(t,x):=f(t,\tau_{x}\omega)$.  Then, for 
  almost every $\omega\in\Omega$, there holds
  $f_{\omega}\in L^{p}(0,T;L_{loc}^{p}(\Rd))$.
   Furthermore, for almost
  every $\omega\in\Omega$, there holds 
  \begin{equation}
    \lim_{\eps\rightarrow0}\int_{0}^{T}\int_{\Q}f(t,\tau_{\frac{x}{\eps}}\omega)\,
    dx\, dt=\left|\Q\right|\int_{0}^{T}\int_{\Omega}f(t,\omega)\,
    d\cP(\omega)\, dt\,.\label{eq:lem-ergodic-time}\end{equation}
\end{lem}
\begin{proof}
  Since the mapping $(x,\omega)\mapsto\tau_{x}\omega$ is measurable,
  we find that $\tilde{f}(\omega,t,x):=f(t,\tau_{x}\omega)$ is
  $\cP\otimes\cL\otimes\cL^{d}$-measurable.  Since the mappings
  $\tau_{x}:\Omega\rightarrow\Omega$ are measure preserving, we find
  for every $x\in\Rd$ \[
  \int_{0}^{T}\int_{\Omega}\left|f(t,\omega)\right|^{p}\,
  d\cP(\omega)dt=\int_{0}^{T}\int_{\Omega}\left|f(t,\tau_{x}\omega)\right|^{p}\,
  d\cP(\omega)dt\,.\] Integrating the last equation over
  $\Q\subset\Rd$ and applying Fubini's theorem, we obtain
  \begin{equation*}
    \left|\Q\right|\int_{0}^{T}\int_{\Omega}\left|f(t,\omega)\right|^{p}\, d\cP(\omega)dt 
    =\int_{\Omega}\int_{0}^{T}\int_{\Q}\left|f(t,\tau_{x}\omega)\right|^{p}\, dx\, dt\, d\cP(\omega)\,.
  \end{equation*}
  Thus, $\tilde{f}$ has the integrability $\tilde{f}\in
  L^{p}(\Omega;L^{p}(0,T;L^{p}(\Q)))$ and $f_{\omega}\in
  L^{p}(0,T;L^{p}(\Q))$ for almost every $\omega\in\Omega$.  In
  particular, $f_{\omega}\in L^{1}(0,T;L^{1}(\Q))$. Setting
  $F(\omega):=\int_{0}^{T}f(t,\omega)\, dt$, we find as a consequence
  of Theorem \ref{thm:ergodic-thm}:\[
  \lim_{\eps\rightarrow0}\int_{0}^{T}\int_{\Q}f(t,\tau_{\frac{x}{\eps}}\omega)\,
  dx\,
  dt=\lim_{\eps\rightarrow0}\int_{\Q}F(\tau_{\frac{x}{\eps}}\omega)\,
  dx=\left|\Q\right|\int_{\Omega}F\,
  d\cP=\left|\Q\right|\int_{0}^{T}\int_{\Omega}f\, d\cP\, dt\,.\]
This was the claim in \eqref{eq:lem-ergodic-time}.
\end{proof}

\begin{lem}
  \label{lem:Div-Curl-Lemma}(Div-curl-lemma) Let $U\subset\Rd$ be open
  and bounded with Lipschitz-boundary $\partial U$. For a sequence
  $\eps\rightarrow0$ we consider sequences of functions $u^\eps$ and
  $\v^\eps$ as follows:
  \begin{align*}
    &u^{\eps}\in L^{2}(0,T;L^{2}(U;\Rdd))\quad\mbox{with}
    \quad\diver u^{\eps}(t)=0\quad\mbox{in }\mathcal{D}'(U)\mbox{ for a.e. }t\in[0,T]\,,\\
    &\v^{\eps}\in L^{2}(0,T;L^{2}(U;\Rdd))\,,\quad
    \v^{\eps}(t,x):=\v(t,\tau_{\frac{x}{\eps}}\omega)\quad\mbox{for}\quad
    \v\in L^{2}(0,T;\cV_{pot}^{2}(\Omega))
  \end{align*}
  and some $\omega\in\Omega$. We assume the boundedness
  $\norm{u^{\eps}}_{L^{2}(0,T;L^{2}(U))} \le C_0$. Then, for almost
  every $\omega\in\Omega$, there holds
  \begin{equation}
    \lim_{\eps\rightarrow0}\int_{0}^{T}\int_{U}u^{\eps}:\v^{\eps}=0\,.
    \label{eq:conv-div-curl}
  \end{equation}
\end{lem}

\begin{proof}
  In this proof, we omit the time-dependence of $u^{\eps}$ and $\v$
  for simplicity of notation, i.e. we consider $u^{\eps}\in
  L^{2}(U;\Rdd)$ and $\v\in \cV_{pot}^{2}(\Omega)$. In the time
  dependent case, one has to apply Lemma \ref{lem:ergodic-time}
  instead of the ergodic theorem \ref{thm:ergodic-thm}.

  We consider a compact set $K\subset U$ and a cut-off function
  $\psi\in C^{\infty}(\Rd)$ with $\psi\equiv1$ on $K$, $\psi\equiv0$
  on $\Rd\backslash U$, and $0\le \psi\le 1$.  We fix
  $\omega\in\Omega$ such that $x\mapsto \v(\tau_{x}\omega)\in
  L_{pot,loc}^{2}(\Rd)$ and such that the assertion of Theorem
  \ref{thm:ergodic-thm} holds.  Furthermore, we make use of
  $\phi_{\eps,\omega,v}$ of Lemma
  \ref{lem:existence-primitive-v-omega} and observe the limit behavior
  \begin{align}
    \label{eq:div-curl-proof-1}
    \int_{U}u^{\eps}:\v^{\eps}\psi&=\int_{U}u^{\eps}:(\nabla
    \phi_{\eps,\omega,v})\psi
    =\int_{U}u^{\eps}:\nabla_x\left(\phi_{\eps,\omega,v}\psi\right)
    -\int_{U}u^{\eps}:\left(\phi_{\eps,\omega,v}\otimes\nabla_x\psi\right)\nonumber\\
    &=-\int_{U}u^{\eps}:\left(\phi_{\eps,\omega,v}\otimes\nabla_x\psi\right)\to
    0
  \end{align}
  as $\eps\rightarrow\infty$ due to $\phi_{\eps,\omega,v}\to 0$ of
  Lemma \ref{lem:existence-primitive-v-omega} and the boundedness of
  $\nabla\psi$ and $u^\eps$.

  Concerning the integral over $u^{\eps}:\v^{\eps} (1-\psi)$, we find
  by the ergodic theorem \ref{thm:ergodic-thm}
  \begin{equation}
    \left|\int_{U}u^{\eps}:\v^{\eps}(1-\psi)\right|\leq C_0
    \norm{\v^{\eps}}_{L^{2}(U\backslash K)}\to C_0 
    \norm{\v}_{\LOM}\left|U\backslash K\right|^{\frac{1}{2}}\label{eq:div-curl-proof-4}
  \end{equation}
  as $\eps\rightarrow0$. Choosing $K\subset U$ large we obtain \eqref
  {eq:conv-div-curl}.
\end{proof}

\subsection{The averaging property of $\Sigma$}

\begin{thm}[Averaging property]
  \label{thm:Sto-averaging} 
  Let the coefficients $B(\omega)$, $C(\omega)$,
  $\Psi(\,\cdot\,;\,\omega)$ be as in Assumption \ref{ass:sto-coeffs}
  and let realizations $C_{\eps}$, $B_{\eps}$, $\Psi_{\eps}$ be
  defined by \eqref {eq:C-eps-from-C}. Then, for a.e.\,$\omega\in
  \Omega$, the coefficients allow averaging in sense of Definition
  \ref{def:Averaging} with the operators $\Sigma$ and $\Pi$ of
  \eqref{eq:def-sig-stoch}.
\end{thm}

\begin{proof}
  We will prove a slightly stronger result: Given $\xi\in
  H_\ast^{1}(0,T;\symM)$, let $\left(p,\, z,\,\v\right)$ be the unique
  solution of \eqref{eq:EPISIG-1}--\eqref{eq:EPISIG-2} (which exists
  by Theorem \ref{thm:E-Pi-Sig-well-def}). Let $\omega\in\Omega$ be
  such that $p_{\omega}(t,x):=p(t,\tau_{x}\omega)$,
  $z_{\omega}(t,x):=z(t,\tau_{x}\omega)$ and
  $\v_{\omega}(t,x):=\v(t,\tau_{x}\omega)$ satisfy \[ p_{\omega}\in
  H^{1}(0,T;L_{loc}^{2}(\Rd;\symM))\,,\,\, z_{\omega}\in
  H^{1}(0,T;L_{sol,loc}^{2}(\Rd))\,,\,\,\v_{\omega}\in
  H^{1}(0,T;L_{pot,loc}^{2}(\Rd))\,.\] This regularity is valid for
  a.e. $\omega$ as can be seen applying Lemma \ref{lem:ergodic-time}
  to time derivatives. Furthermore, we choose $\omega$ as in
  Assumption \ref{ass:sto-coeffs}. For any $\eps>0$ let
  $\tilde{p}^{\eps}(t,x):=p\left(t,\tau_{\frac{x}{\eps}}\omega\right)$,
  $\tilde{z}^{\eps}(t,x):=z\left(t,\tau_{\frac{x}{\eps}}\omega\right)$,
  $\tilde{\v}^{\eps}(t,x):=\v\left(t,\tau_{\frac{x}{\eps}}\omega\right)$
  be realizations.  Let $\cT\subset\Rd$ be a simplex and let
  $u^{\eps}$, $p^{\eps}$, $\sigma^{\eps}$ be the unique solution
  to \begin{align}
    -\diver\sigma^{\eps} & =0\,,\nonumber \\
    \nabla^{s}u^{\eps} & =C_{\eps}\sigma^{\eps}+p^{\eps}\label{eq:EPISIG-conv-1}\\
    \partial_{t}p^{\eps} &
    \in\partial\Psi_{\eps}(\sigma^{\eps}-B_{\eps}p^{\eps}\,;\,.\,)\nonumber \end{align}
  on $\cT$ with boundary condition
  \begin{equation}
    \label {eq:proof-4-3-BC}
    u^{\eps}(x)=\xi\cdot x\qquad\mbox{on }\partial\cT
  \end{equation}
  and initial condition $p^{\eps}(0,\cdot)=0$ (we recall
  $\partial\Psi_{\eps}(\sigma\,;\, x) :=\partial \Psi
  (\sigma\,;\,\tau_{\frac{x}{\eps}}\omega)$). We will prove that the
  realizations of the stochastic cell solutions and the plasticity
  solutions on $\cT$ coincide in the limit $\eps\to 0$; more
  precisely, we claim that
  \begin{equation}
    \lim_{\eps\rightarrow0}\left(\norm{\sigma^{\eps}-
        \tilde{z}^{\eps}}_{L^{2}(0,T;L^{2}(\cT))}+
      \norm{p^{\eps}-\tilde{p}^{\eps}}_{L^{2}(0,T;L^{2}(\cT)}\right)=0\,.
    \label{eq:EPISIG-conv-stronger}
  \end{equation}
  
  \smallskip Let us first show that \eqref {eq:EPISIG-conv-stronger}
  indeed implies Theorem \ref {thm:Sto-averaging}: The ergodic theorem
  in the version of Lemma \ref{lem:ergodic-time} and the definition of
  $\Sigma$ and $\Pi$ in \eqref{eq:def-sig-stoch} imply that
  $\fint_{\cT}\tilde{z}^{\eps}(.)\to \int_{\Omega}z(.) =
  \Sigma(\xi)(.)$ and $\fint_{\cT}\tilde{p}^{\eps}(.)\to
  \int_{\Omega}p(.) = \Pi(\xi)(.)$ holds in the space
  $L^2(0,T;\symM)$. Equation \eqref{eq:EPISIG-conv-stronger} therefore
  yields $\fint_{\cT}\sigma^{\eps}\to \Sigma(\xi)$ and
  $\fint_{\cT}p^{\eps}\rightarrow\Pi(\xi)$ in $L^2(0,T;\symM)$.  This
  provides the averaging property \eqref {eq:averaging-conv-p-sig} of
  Definition \ref {def:Averaging} (at first, for a subsequence
  $\eps\to 0$ for almost every $t\in (0,T)$, then, since the limit is
  determined, along the original sequence $\eps\to 0$).

  \smallskip Let us now prove \eqref{eq:EPISIG-conv-stronger}. We will
  use a testing procedure and energy-type estimates. Due to
  \eqref{eq:EPISIG-1}--\eqref{eq:EPISIG-2}, $\tilde{z}^{\eps}$,
  $\tilde{p}^{\eps}$ and $\tilde{\v}^{\eps}$ satisfy the following
  system of equations on $\cT\times (0,T)$
  \begin{align}
    -\diver\tilde{z}^{\eps} & =0\,,\nonumber \\
    \xi & =C_{\eps}\tilde{z}^{\eps}+\tilde{p}^{\eps}-\left(\tilde{\v}^{\eps}\right)^s\,,\label{eq:EPISIG-conv-2}\\
    \partial_{t}\tilde{p}^{\eps} &
    \in\partial\Psi_{\eps}(\tilde{z}^{\eps}-B_{\eps}\tilde{p}^{\eps}\,;\,.\,)\,.\nonumber 
  \end{align}
  In what follows we use the notation $\left| \zeta\right|
  _{B_{\eps}}^{2}:=\zeta:B_{\eps}\zeta$ and $\left| \zeta\right|
  _{C_{\eps}}^{2}:=\zeta:C_{\eps}\zeta$. We take the difference of
  \eqref{eq:EPISIG-conv-1}$_{1}$ and \eqref{eq:EPISIG-conv-2}$_{1}$,
  multiply the result by
  $\left(\partial_{t}u^{\eps}-\partial_{t}\left(\xi\cdot
      x\right)\right)$ and integrate over $\cT$. We integrate by parts
  and exploit that boundary integrals vanish due to \eqref
  {eq:proof-4-3-BC},
  \begin{align}
    0 & =-\int_{\cT}\left(\tilde{z}^{\eps}-\sigma^{\eps}\right) :
    \left(\partial_{t}\nabla^{s}u^{\eps}-\partial_{t}\xi\right)\nonumber \\
    & =\int_{\cT}\left(\tilde{z}^{\eps}-\sigma^{\eps}\right):
    \partial_{t}\left(C_{\eps}\tilde{z}^{\eps}+\tilde{p}^{\eps}-\left(\tilde{\v}^{\eps}\right)^s - C_{\eps}\sigma^{\eps}-p^{\eps}\right)\nonumber \\
    & =\frac{1}{2}\frac{d}{dt}\int_{\cT}\left[\left(\tilde{z}^{\eps}-\sigma^{\eps}\right):\left(C_{\eps}\left(\tilde{z}^{\eps}-\sigma^{\eps}\right)\right)+\left(\tilde{p}^{\eps}-p^{\eps}\right):\left(B_{\eps}\left(\tilde{p}^{\eps}-p^{\eps}\right)\right)\right]+\int_{\cT}(\tilde{z}^{\eps}-\sigma^{\eps}):\partial_{t}\tilde{\v}^{\eps}\displaybreak[2]\nonumber \\
    & \phantom{=}+\int_{\cT}\left(\partial_{t}\tilde{p}^{\eps}-\partial_{t}p^{\eps}\right):\left(\left(\tilde{z}^{\eps}-B_{\eps}\tilde{p}^{\eps}\right)-\left(\sigma^{\eps}-B_{\eps}p^{\eps}\right)\right)\,.\displaybreak[2]\nonumber \\
    & \in\frac{1}{2}\frac{d}{dt}\int_{\cT}\left[\left| \tilde{z}^{\eps}-\sigma^{\eps}\right| _{C_{\eps}}^{2}+\left| \tilde{p}^{\eps}-p^{\eps}\right| _{B_{\eps}}^{2}\right]+\int_{\cT}(\tilde{z}^{\eps}-\sigma^{\eps}):\partial_{t}\tilde{\v}^{\eps}\nonumber \\
    &
    \phantom{=}+\int_{\cT}\left(\partial\Psi_{\eps}\left(\tilde{z}^{\eps}-B_{\eps}\tilde{p}^{\eps}\right)-\partial\Psi_{\eps}\left(\sigma^{\eps}-B_{\eps}p^{\eps}\right)\right):\left(\left(\tilde{z}^{\eps}-B_{\eps}\tilde{p}^{\eps}\right)-\left(\sigma^{\eps}-B_{\eps}p^{\eps}\right)\right)\,.\label{eq:EPISIG-conv-3}
  \end{align}
  In the second line, we used \eqref{eq:EPISIG-conv-1}$_{2}$ and
  \eqref{eq:EPISIG-conv-2}$_{2}$.  In the third line we used the
  symmetry of $\sigma^{\eps}$ and $\tilde{z}^{\eps}$ to replace
  $(\tilde{\v}^{\eps})^{s}$ by $\tilde{\v}^{\eps}$.

  Concerning the second integral on the right hand side of
  \eqref{eq:EPISIG-conv-3}, note that $\int_0^t
  \int_{\cT}\tilde{z}^{\eps}:\partial_{t}\tilde{\v}^{\eps}\to \int_0^t
  \int_{\cT}\int_{\Omega}z:\partial_{t}\v = 0$ by Lemma
  \ref{lem:ergodic-time} and orthogonality of $L_{sol}^{2}(\Omega)$
  and $\cV_{pot}^{2}(\Omega)$. Furthermore,
  $\int_0^t\int_{\cT}\sigma^{\eps}:\partial_{t}\tilde{\v}^{\eps}\rightarrow0$
  by Lemma \ref{lem:Div-Curl-Lemma}. By monotonicity of
  $\partial\Psi_{\eps}$, the last integral on the right hand side of
  \eqref{eq:EPISIG-conv-3} is positive. An integration over $(0,t)$
  therefore provides
  \begin{equation}
    \limsup_{\eps\rightarrow0}\int_{\cT}\left[\left| \tilde{z}^{\eps}-\sigma^{\eps}\right|
      _{C_{\eps}}^{2}+\left| \tilde{p}^{\eps}-p^{\eps}\right|
      _{B_{\eps}}^{2}\right](t)\leq\limsup_{\eps\rightarrow0} \int_0^t\int_{\cT}
    (\tilde{z}^{\eps}-\sigma^{\eps}) \del_t \tilde{\v}^{\eps} = 0
    \,,\label{eq:EPISIG-conv-4}
  \end{equation} 
  where we used that initial data vanish, $\tilde{z}^{\eps}|_{t=0} =
  0$ by \eqref {eq:EPISIG-conv-2} and $\sigma^{\eps}|_{t=0} = 0$ by
  \eqref {eq:EPISIG-conv-1} for vanishing $p^\eps$ and $\xi$ in $t=0$.
  We have thus shown \eqref {eq:EPISIG-conv-stronger} and hence
  Theorem \ref {thm:Sto-averaging}.
\end{proof}

\subsection{Admissibility of $\Sigma$}

\begin{thm}[Admissibility]
  \label{thm:existence-limit-stochastic}
  Let the coefficients $B(\omega)$, $C(\omega)$,
  $\Psi(\,\cdot\,;\,\omega)$ and data $\U$, $f$ be as in Assumption
  \ref{ass:sto-coeffs}. Then the causal operator $\Sigma$ of
  Definition \ref {def:Sigma} satisfies the sufficient condition for
  admissibility of Definition \ref {def:admissible-operator}.
\end{thm}

\begin{proof}
  We have to study solutions $u_h$ of the discretized effective
  problem with the discretized boundary data $\U_{h}\rightarrow\U$
  strongly in $H^{1}(0,T;H^{1}(\Q))$ as $h\rightarrow0$. With $\Sigma$
  given through \eqref{eq:def-sig-stoch}, let $u_{h}\in
  H^{1}(0,T;H^{1}(\Q))$ be a sequence with $u_{h}\in \U_{h} +
  H^{1}(0,T;Y_{h})$, satisfying the discrete system
  \begin{equation}
    \int_0^T\int_{\Q}\Sigma(\nabla^{s}u_{h}):\nabla\varphi
    = \int_0^T\int_{\Q}f\cdot\varphi\qquad\forall\varphi\in
    L^{2}(0,T;Y_{h})\,.\label{eq:sto-discrete-eff}
  \end{equation}
  We furthermore have the weak convergence $u_{h}\weakto u\in
  H^{1}(0,T;H^{1}(\Q;\Rd)$ as $h\rightarrow0$ for some $u\in \U +
  H^{1}(0,T;H_{0}^{1}(\Q;\Rd))$.  Our aim is to show that $u$ solves
  the effective problem
  \begin{equation}
    \int_0^T\int_{\Q}\Sigma(\nabla^{s}u):\nabla\varphi
      =\int_0^T\int_{\Q}f\cdot\varphi\qquad\forall\varphi\in
      L^{2}(0,T;H_{0}^{1}(\Q))\,.\label{eq:sto-discrete-hom-eff}
  \end{equation}

  \smallskip {\em Step 1.}  For every $x\in\Q$, we denote by
  $p_{h}(t,x,\cdot)$, $z_{h}(t,x,\cdot)$, $\v_{h}(t,x,\cdot)$ the
  solutions of \eqref{eq:EPISIG-1}--\eqref{eq:EPISIG-2} corresponding
  to $\xi(t)=\nabla^{s}u_{h}(t,x)$. By definition of $\Sigma$, there
  holds $\Sigma(\nabla^{s}u_{h}) = \int_{\Omega} z_{h}(\omega)\,
  d\cP(\omega)$. The a priori estimate of Theorem
  \ref{thm:E-Pi-Sig-well-def} provides \[
  \norm{p_{h}}_{\cV^{1}_{0,0}}+\norm{z_{h}}_{\cV^{1}_{0,0}}
  +\norm{\v_{h}}_{\cV^{1}_{0,0}}\leq C\norm {\nabla^s
    u}_{H^{1}(0,T;L^{2}(\Q))}\,,\] where
  $\cV^{1}_{0,0}:=H^{1}(0,T;L^{2}(\Q;L^{2}(\Omega;\Rdd)))$.  By this
  estimate, we obtain the weak convergence in $(\cV^{1}_{0,0})^3$ of a
  subsequence, again denoted $(p_{h}, z_{h}, \v_{h})$, weakly
  converging to some limit $(p, z, \v)$. The limit satisfies again the
  linear law \eqref{eq:EPISIG-2},
  \begin{equation} 
    C z = \nabla^{s}u + \v^s - p\,.\label{eq:approx-prop-eq-1}
  \end{equation}
  Equation \eqref{eq:sto-discrete-eff} can be rewritten as\[
  \int_{0}^{T}\int_{\Q}\int_{\Omega}z_{h}(t,x,\omega)\, d\cP(\omega)
  :\,\nabla\varphi(x)\,
  dx=\int_{0}^{T}\int_{\Q}f\cdot\varphi\qquad\forall\varphi\in
  L^{2}(0,T;Y_{h})\,,\] and the limit $h\to 0$
  provides
  \begin{equation}
    \int_{0}^{T}\int_{\Q}\int_{\Omega}z:\nabla\varphi
    =\int_{0}^{T}\int_{\Q}f\cdot\varphi\qquad\forall\varphi\in
    L^{2}(0,T;H_{0}^{1}(\Q))\,.\label{eq:approx-prop-eq-2}
  \end{equation}

  \smallskip {\em Step 2.} It remains to verify $\int_{\Omega}z =
  \Sigma(\nabla^s u)$.  We use
  $\varphi=\partial_{t}\left(u_{h}-\U_{h}\right)$ as a test function
  in \eqref{eq:sto-discrete-eff} and exploit the orthogonality $0 =
  \int_{\Q}\int_{\Omega}z_{h}:\partial_{t}\v_{h}$.  We follow the
  lines of the calculation in \eqref{eq:EPISIG-est-1-a} to obtain
  \begin{align}
    &\int_{\Q}f\cdot\partial_{t}\left(u_{h}-\U_{h}\right)
    +\int_{\Q}\int_{\Omega}z_{h}:\nabla\partial_{t}\U_{h}    \nonumber\\
    &\qquad =\int_{\Q}\int_{\Omega}z_{h}:\partial_{t}\nabla^{s}u_{h}
    =\int_{\Q}\int_{\Omega}\left[z_{h}:C\partial_{t}z_{h}+z_{h}:\partial_{t}p_{h}-z_{h}:\partial_{t}\v_{h}\right]
    \nonumber\\
    &\qquad
    =\frac{1}{2}\frac{d}{dt}\left(\int_{\Q}\int_{\Omega}p_{h}:Bp_{h}+\int_{\Q}\int_{\Omega}z_{h}:Cz_{h}\right)
    +\int_{\Q}\int_{\Omega}\left(\Psi^{\ast}\left(\partial_{t}p_{h}\right)+\Psi\left(z_{h}-Bp_{h}\right)\right)\,.
    \label{eq:approx-prop-ineq-1}
  \end{align}
  Taking weak limits in \eqref{eq:approx-prop-ineq-1}
  yields
  \begin{multline*}
    \int_{0}^{T}\int_{\Q}\int_{\Omega}\left(\Psi^{\ast}\left(\partial_{t}p\right)+\Psi\left(z-Bp\right)\right)\\
    \leq\int_{0}^{T}\int_{\Q}f\cdot\partial_{t}\left(u-\U\right)
    +\int_{0}^{T}\int_{\Q}\int_{\Omega}z:\nabla\partial_{t}\U
    -\frac{1}{2}\left.\left(\int_{\Q}\int_{\Omega}p:(Bp)+\int_{\Q}\int_{\Omega}z:(Cz)\right)\right|_{0}^{T}\,.
  \end{multline*}

  Relations \eqref{eq:approx-prop-eq-1} and
  \eqref{eq:approx-prop-eq-2} allow to perform the calculations of
  \eqref {eq:approx-prop-ineq-1} also for the limit functions. We
  obtain from the last inequality\[
  \int_{0}^{T}\int_{\Q}\int_{\Omega}\left(\Psi^{\ast}\left(\partial_{t}p\right)
  +\Psi\left(z-Bp\right)\right)\leq\int_{0}^{T}\int_{\Q}\int_{\Omega}\partial_{t}p:(z-Bp)\,.\]
  The Fenchel inequality of Lemma \ref{lemma:convprop} (iii) yields
  $\partial_{t}p:(z-Bp)\leq
  \Psi^{\ast}\left(\partial_{t}p\right)+\Psi\left(z-Bp\right)$
  pointwise. We can therefore conclude from the Fenchel equality
  \begin{equation}
    \partial_{t}p\in\partial\Psi(\sigma-Bp)\,.\label{eq:lim-inclusion}
  \end{equation}

  Relations \eqref{eq:approx-prop-eq-1} and \eqref{eq:lim-inclusion}
  imply that $z$ is defined as in the definition of $\Sigma$, hence
  $\int_{\Omega}z(t,x,\,.\,) = \Sigma(\nabla^{s}u)(t,x,\,.\,)$ for
  every $t\in[0,T]$ and a.e. $x\in\Q$. Therefore,
  \eqref{eq:approx-prop-eq-2} is equivalent with
  \eqref{eq:sto-discrete-hom-eff} and the theorem is shown.
\end{proof}

\subsection{Conclusion of the proof}
\label{ssec:Proof-of-Main-Theorem} 

We can now conclude the proof of our main result, Theorem \ref
{thm:Main-Theorem}.  Theorem \ref{thm:existence-limit-stochastic}
implies that $\Sigma$ of \eqref{eq:def-sig-stoch} is
admissible. Theorem \ref{thm:Sto-averaging} yields that, for almost
every $\omega\in\Omega$, the coefficients $C_{\eps,\omega}(x),\,
B_{\eps,\omega}(x),\,\Psi_{\eps,\omega}(\sigma;x)$ allow averaging
with limit operator $\Sigma$.  We can therefore apply Theorem
\ref{thm:needle-thm} and obtain
\begin{align*}
  & u^{\eps}\weakto u\quad\mbox{weakly in }H^{1}(0,T;H^{1}(\Q;\Rd))\\
  & p^{\eps}\weakto\Pi(\nabla^{s}u),\quad\sigma^{\eps}\weakto\Sigma(\nabla^{s}u)
  \quad\mbox{weakly in }H^{1}(0,T;L^{2}(\Q;\Rdd))\,,
\end{align*} 
where $u$ is the unique weak solution to the homogenized problem\[
-\diver\Sigma(\nabla^{s}u)=f\,\] with boundary condition $\U$ as in
Definition \ref{def:limit-problem}.  Theorem \ref {thm:Main-Theorem}
is shown.

\appendix

\section{An example for the stochastic setting}

Our aim here is to describe briefly a non-trivial example for a
stochastic setting: the checker board construction of i.i.d. random
variables.  Our main goal is to show that the compactness assumption
on $\Omega$ is not too restrictive and still permits the analysis of
interesting problems.

We use $Y:=[0,1[^{d}$ with the topology of the torus and the partition
of $\Rr^{d}$ with unit cubes $\cC_{z}:=z+Y$ for $z\in\Zz^{d}$. We
consider the sets
\begin{align*}
  \tilde{\Omega} & :=\left\{ u\in L^{\infty}(\Rr^{d})\,|\, u|_{\cC_{z}}\equiv c_{z}\,,
    \mbox{ for some }\ c: \Zz^{d}\to [0,1], z\mapsto c_z\right\} \\
  \Omega & :=\left\{ u\in L^{\infty}(\Rr^{d})\,|\,\exists\xi\in
    Y\:\mbox{s.t. }u(.-\xi)\in\tilde{\Omega}\right\} \,.
\end{align*}
For $u\in\Omega$ we denote a shift $\xi$ from the above definition as
$\xi(u)$. Since $L^{1}(\Rr^{d})$ is separable, we infer from
\cite{Brezis1983a}, Theorem III.28, that $L^{\infty}(\Rr^{d})$ with
the weak-$\ast$-topology is metrizable: With a countable and dense
subset $\left(\phi_{i}\right)_{i\in\Nn}$ of $L^{1}(\Rr^{d})$, a metric
$d$ on $B_{\infty}$ is given by \[
d(u,v):=\sum_{i=1}^{\infty}\frac{1}{2^{i}}\left|\left\langle
    u-v,\phi_{i}\right\rangle \right|\,.\] We infer that
$B_{\infty}:= \overline{B_{1}(0)}\subset L^{\infty}(\Rr^{d})$ with the
weak-$\ast$-topology is a compact metric space. The sets
$\tilde{\Omega}$ and $\Omega$ are closed subsets of
$\left(B_{\infty},d\right)$ and thus compact metric spaces.

The probability measure on $\Omega$ corresponding to i.i.d.\,random
variables can be defined with the help of elementary subsets. For an
open set $U\subseteq Y$, a number $k\in\Nn$, and relatively open
intervals $I_{z}:=\left((a_{z},b_{z})\cap[0,1]\right)\subset[0,1]$,
$z\in\Zz^{d}$ and $a_{z}<b_{z}$, the sets \begin{equation}
  A\left(U,(I_{z})_{z\in\Zz^{d}},k\right)=\left\{
    u\in\Omega\,|\,\xi(u)\in U\,,\,\, u(.-\xi(u))|_{\cC_{z}}\in
    I_{z}\,\forall z\,,\,\,|z|\leq
    k\right\} \label{eq:weak-star-basis}\end{equation} are open and
form a basis of the topology in $\Omega$. For any such set $A(\,.\,)$
we define\[
\cP\left(A\left(U,(I_{z})_{z\in\Zz^{d}},k\right)\right):=|U|\prod_{|z|\leq
  k}|b_{z}-a_{z}|\,.\]

We finally introduce $\tau_{x}:\Omega\rightarrow\Omega$ for every
$x\in\Rr^{d}$ through $\tau_{x}u(\,.\,)=u(x+.)$. It is easy to check
that the family $\left(\tau_{x}\right)_{x\in\Rr^{d}}$ is a dynamical
system. Since $\cP(A)=\cP(\tau_{x}A)$ for $A$ as in
\eqref{eq:weak-star-basis} and $x\in\Rr^{d}$, the dynamical system is
measure preserving.

\bibliographystyle{abbrv}
\bibliography{lit-needleplasticity-2}

\end{document}